\documentclass{amsart}

\usepackage{amssymb}
\usepackage{amsthm}
\usepackage{wasysym}
\usepackage[numbers]{natbib}

\newcommand{\abs}[1]{\left|#1\right|}
\newcommand{\norm}[1]{\left\| #1 \right\|}
\newcommand{\dual}[2]{\left\langle #1 , #2\right\rangle}
\newcommand{\A}{\mathcal{A}}
\newcommand{\xnice}{\mathbf{x}}
\newcommand{\ynice}{\mathbf{y}}
\newcommand{\Ho}{\mathcal{H}}

\newcommand{\Veps}{{\mathcal{V}_H^R}}
\newcommand{\Ueps}{{\mathcal{U}_H^R}}
\newcommand{\khat}{\hat{k}}
\newcommand{\ptilde}{\tilde{p}}
\newcommand{\R}{\mathbb{R}}
\newcommand{\C}{\mathbb{C}}
\newcommand{\N}{\mathbb{N}}

\DeclareMathOperator{\Div}{div}
\DeclareMathOperator{\Ext}{Ext}
\DeclareMathOperator{\id}{id}

\newtheorem{Theorem}{Theorem}
\newtheorem{Proposition}[Theorem]{Proposition}
\newtheorem{Lemma}[Theorem]{Lemma}

\begin{document}

\title[Decay via viscoelastic boundary damping]{On the decay rate for the wave equation with viscoelastic boundary damping}
\author[Reinhard Stahn]{Reinhard Stahn}

\begin{abstract} 
We consider the wave equation with a boundary condition of memory type. Under natural conditions on the acoustic impedance $\khat$ of the boundary one can define a corresponding semigroup of contractions \cite{DFMP2010a}. With the help of Tauberian theorems we establish energy decay rates via resolvent estimates on the generator $-\A$ of the semigroup. We reduce the problem of estimating the resolvent of $-\A$ to the problem of estimating the resolvent of the corresponding stationary problem. Under not too strict additional assumptions on $\khat$ we establish an upper bound on the resolvent. For the wave equation on the interval or the disk or for certain acoustic impedances making $0$ a spectral point of $\A$ we prove our estimates to be sharp.
\end{abstract}

\maketitle

{\let\thefootnote\relax}
{\let\thefootnote\relax}


\section{Introduction}\label{sec: Introduction}
Let $\Omega\subset\R^d$ be a bounded domain with Lipschitz boundary and $k:\R\rightarrow[0,\infty)$ be an integrable function, depending on the time-variable only and vanishing on $(-\infty,0)$. We consider a model for the reflection of sound on a wall \cite{PrussProbst1996}: 
\begin{equation}\label{eq: wave equation}
 \begin{cases}
  U_{tt}(t,x) - \Delta U(t,x) = 0 &  (t\in\R, x\in\Omega), \\
  \partial_n U(t,x) + k* U_t(t,x) = 0 & (t\in\R, x\in\partial\Omega).
 \end{cases}
\end{equation}
The function $U$ is called the \textit{velocity potential}. One can derive the acoustic pressure $p(t,x)=U_t(t,x)$ and fluid velocity $v(t,x)=-\nabla U(t,x)$ from $U$. The second formula gives the velocity potential its name. The convolution is defined by the usual formula $k*U_t(t,x)=\int_0^{\infty} k(r)U_t(t-r,x) dr$. Here $n$ is the outward normal vector of $\partial\Omega$, which exists almost everywhere for Lipschitz domains. Furthermore $\partial_n$ denotes the normal derivative on the boundary.

We assume that $k\in L^1(0,\infty)$ is a completely monotonic function\footnote{We use the convention to identify functions defined on the interval $[0,\infty)$ with functions defined on $\R$ but zero to the left of $t=0$.}. That is, there exists a positive Radon measure $\nu$ on $[0,\infty)$ such that $k(t)=\int_{[0,\infty)}e^{-\tau t} d\nu(\tau)$. We note here that the integrability assumption on $k$ is easily checked to be equivalent to 
\begin{equation}\label{eq: k is integrable}
 \nu(\{0\})=0 \text{ and } \int_0^{\infty}\tau^{-1} d\nu(\tau)<\infty.
\end{equation}
Let $e_{\tau}(t)=e^{-\tau t} 1_{[0,\infty)}(t)$ and
\begin{align*}
 \psi(t,\tau,x) = e_{\tau} * U_t(t,x) \quad (t\in\R, \tau\geq0, x\in\partial\Omega) .
\end{align*}
Informally it is not difficult to see that (\ref{eq: wave equation}) for $t>0$ in conjunction with the information that $p=U_t$ and $v=-\nabla U$ at time $t=0$ and $\int_0^{\infty}k(t)U_t(-t)dt$ at the boundary (the \emph{``essential''} data from the past) is equivalent to
\begin{equation}\label{eq: physical wave equation}
 \begin{cases}
  p_t(t,x) + \Div v(t,x) = 0 &  (t>0, x\in\Omega), \\
  v_t(t,x) + \nabla p(t,x) = 0 & (t>0, x\in\Omega), \\
  [\psi_t+\tau\psi-p](t,\tau,x) = 0 & (t>0, \tau>0, x\in\partial\Omega), \\
  \left[-v\cdot n + \int_0^{\infty} \psi(\tau) d\nu(\tau)\right](t,x) = 0 & (t>0, x\in\partial\Omega),
 \end{cases}
\end{equation}
and the information of the initial state $\xnice_0=(p_0, v_0, \psi_0)$ of the system at time $t=0$. It is important to observe that $p_0$ and $v_0$ cannot fully describe the system's state at $t=0$ since there are memory effects at the boundary. The missing data from the past is stored in the auxiliary function $\psi$.

Let us define the energy of the system to be the sum of potential, kinetic and boundary energy:
\begin{align*}
  E(\xnice_0) = \int_{\Omega} \abs{p_0(x)}^2 + \abs{v_0(x)}^2 dx + \int_0^{\infty}\int_{\partial\Omega} \abs{\psi_0(\tau,x)}^2 dS(x) d\nu(\tau) .
\end{align*}
Furthermore we introduce the homogeneous first order energy by
\begin{align*}
  E_1^{hom}(\xnice_0) =  \int_{\Omega} \abs{\nabla p_0}^2 + \abs{\Div v_0}^2 dx + \int_0^{\infty}\int_{\partial\Omega} \abs{\tau\psi_0-p_0}^2 dS d\nu(\tau).
\end{align*}
The first order energy is defined by $E_1=E+E_1^{hom}$. Let us define the (zeroth order) energy space, and the first order energy space by
\begin{align}
 \label{eq: energy space}
 \Ho &= \Ho_0= L^2(\Omega) \times \nabla H^1(\Omega) \times L^2_{\nu}((0,\infty)_{\tau}; L^2(\partial\Omega)) , \\
 \label{eq: first order energy space}
 \Ho_1 &= \{\xnice_0\in\Ho: E_1(\xnice_0)<\infty \text{ and } \left[-v_0\cdot n|_{\partial\Omega} + \int_0^{\infty} \psi_0(\tau) d\nu(\tau)\right] = 0\} .
\end{align}
Here $\nabla H^1(\Omega)$ is the space of vector fields $v\in (L^2(\Omega))^d$ for which there exists a function (potential) $U\in H^1(\Omega)$ such that $v=-\nabla U$. We note that the space of gradient fields $\nabla H^1(\Omega)$ is a closed subspace of $(L^2(\Omega))^d$ since $\Omega$ satisfies the Poincar\'e inequality\footnote{Poincar\'e inequality: If $\Omega$ is a bounded Lipschitz domain then there exists a $C>0$ such that for all $p\in H^1(\Omega)$ with  $\int_{\Omega}p=0$ we have $\int_{\Omega}\abs{p}^2 \leq C\int_{\Omega}\abs{\nabla p}^2$.}. To make the boundary condition, appearing in the definition of $\Ho_1$, meaningful we use that the trace operator $\Gamma:H^1(\Omega)\rightarrow H^{1/2}(\partial\Omega), u\mapsto u|_{\partial\Omega}$ is continuous and has a continuous right inverse. Therefore we see that $v\cdot n |_{\partial\Omega}$ is well defined as an element of $H^{-1/2}(\partial\Omega)=(H^{1/2}(\partial\Omega))^*$ for vector fields $v\in (L^2(\Omega))^d$ with $\Div v\in L^2(\Omega)$ by the relation
\begin{equation}\label{eq: vn}
 \dual{v\cdot n}{\Gamma u}_{H^{-\frac{1}{2}}\times H^{\frac{1}{2}}(\partial\Omega)} 
 = \int_{\Omega} \Div v \overline{u} + \int_{\Omega} v\cdot\nabla \overline{u}
\end{equation}
for all $u\in H^1(\Omega)$. Also note that $E_1(\xnice_0)<\infty$ implies $\psi_0\in L^1_{\nu}$ since $\psi_0(\tau)=\frac{\psi_0(\tau)}{1+\tau}+\frac{\Gamma p}{1+\tau}+\frac{\tau\psi_0(\tau)-\Gamma p}{1+\tau}$ and $(\tau\mapsto \frac{1}{1+\tau})\in L^1_{\nu}\cap L^2_{\nu}(0,\infty)$ by (\ref{eq: k is integrable}).\footnote{Here and in the following we abbreviate $L^p_{\nu}((0,\infty)_{\tau}; L^2(\partial\Omega))$ simply by $L^p_{\nu}$ for $p\in\{1,2\}$.} The quadratic forms $E$ and $E_1$ turn $\Ho$ and $\Ho_1$ into Hilbert spaces respectively.

An initial state $\xnice_0$ is called \emph{classical} if its first order energy is finite and the boundary condition is satisfied (i.e. $\xnice_0\in\Ho_1$). We say that $\xnice\in C^1([0,\infty); \Ho)\cap C([0,\infty);\Ho_1)$ is a \emph{(classical) solution} of (\ref{eq: physical wave equation}) if it satisfies the first two lines in the sense of distributions and the last two lines in the trace sense, i.e. with $v\cdot n$ defined by (\ref{eq: vn}) and $p$ replaced by $\Gamma p$. From Theorem \ref{thm: waves well posed} below plus basics from the theory of $C_0$-semigroups it follows that the initial value problem corresponding to (\ref{eq: physical wave equation}) is well-posed in the sense that for all classical initial data $\xnice_0\in\Ho_1$ there is a unique solution $\xnice$ with $\xnice(0)=\xnice_0$ and the mapping $\Ho_j\owns \xnice_0\mapsto \xnice\in C([0,\infty);\Ho_j)$ is continuous for $j\in\{0,1\}$. For a solution $\xnice$ with $\xnice_0=\xnice(0)$ we also write e.g. $E(t,\xnice_0)$ instead of $E(\xnice(t))$. Note that $E_1^{hom}(\xnice(t)) = E(\dot{\xnice}(t))$ - this justifies the adjective ``homogeneous'' for the quadratic form $E_1^{hom}$.

Our aim is to find the optimal decay rate of the energy, uniformly with respect to classical initial states. This means that we want to find the smallest possible decreasing function $N:[0,\infty)\rightarrow[0,\infty)$ such that
\begin{align*}
 E(t,\xnice_0) \leq N(t)^2 E_1(\xnice_0)
\end{align*}
for all $\xnice_0\in\Ho_1$. Because of Theorem \ref{thm: Batty-Duyckaerts}, \ref{thm: Borichev-Tomilov} and \ref{thm: Martinez} this is essentially equivalent to estimate the resolvent of the wave equation's generator $\A$ (defined in Section \ref{sec: Semigroup approach} below) along the imaginary axis near infinity and near zero. 

Our two main results are Theorem \ref{thm: Preliminary estimates} and \ref{thm: upper resolvent estimate}. The Sections \ref{sec: Correspondence between resolvents} and \ref{sec: upper resolvent estimate} are devoted to the proofs. We illustrate the application of our main results to energy decay by several examples in Section \ref{sec: examples}. Our first main result (Theorem \ref{thm: Preliminary estimates}) implies in particular that the task of estimating the resolvent of the complicated $3\times 3$-matrix operator $\A$ is equivalent to estimate the resolvent of the corresponding (and much simpler) stationary operator. Our second main result (Theorem \ref{thm: upper resolvent estimate}) thus determines an upper resolvent estimate of $\A$ at infinity. Unfortunately we need \emph{additional assumptions} on the acoustic impedance (see (\ref{eq: additional assumptions})). However in our separate treatment of the $\Omega=(0,1)$-setting in Section \ref{sec: 1D case} we see that in this case actually no additional assumptions are required for the conclusion of Theorem \ref{thm: upper resolvent estimate} to hold. Even more is true: The given upper bound on the resolvent is also optimal in the 1D setting. This and observations from the examples lead us to three questions and corresponding conjectures formulated in Section \ref{sec: Conclusion}.

In Section \ref{sec: Semigroup approach} we recall the semigroup approach from \cite{DFMP2010a}. For convenience of the reader we recall some basic and some not so basic facts from the literature concerning the trace operator, fractional Sobolev spaces and Besov spaces in Appendix \ref{apx: Traces}. In Appendix \ref{apx: semiuniform stability} we recall some Batty-Duyckaerts type Theorems. For the reader to be interested in the physical background of equation (\ref{eq: wave equation}) we recommend \cite{IngardMorse}.

\section{The semigroup approach}\label{sec: Semigroup approach}
We reformulate (\ref{eq: physical wave equation}) as an abstract Cauchy problem in a Hilbert space:
\begin{equation}\label{CP}
 \begin{cases}
  \dot{\xnice}(t) + \A \xnice(t) = 0 , \\
  \xnice(0) = \xnice_0 \in \Ho .
 \end{cases}
\end{equation}
Following the approach of \cite{DFMP2010a} we define the energy/state space $\Ho$ as in (\ref{eq: energy space}) and write $\xnice=(p,v,\psi)$ for its elements (the states). Again let $\Gamma:H^1(\Omega)\rightarrow H^{1/2}(\partial\Omega), u\mapsto u|_{\partial\Omega}$ be the trace operator on $\Omega$. By abuse of notation let $\tau$ denote the multiplication operator on $L^2_{\nu}(0,\infty)$ mapping $\psi(\tau)$ to $\tau\psi(\tau)$. We define the wave operator by
\begin{align*}
 \A = \left(
 \begin{array}{ccc}
  0       & \Div & 0 \\
  \nabla  & 0    & 0 \\
  -\Gamma & 0    & \tau
 \end{array}\right)
 \text{ with }
 D(\A) = \Ho_1 .
\end{align*}
Note that $E_1(\xnice_0) = \norm{\xnice_0}_{D(\A)}^2 = \norm{\xnice_0}_{\Ho}^2 + \norm{\A \xnice_0}_{\Ho_1}^2$ for all $\xnice_0\in D(\A)$. 
\begin{Theorem}[\cite{DFMP2010a}]\label{thm: waves well posed}
 The Cauchy problem (\ref{CP}) is well posed. More precisely $-\A$ is the generator of a $C_0$-semigroup of contractions in $\Ho$.
\end{Theorem}

Taking formal Laplace transform of the wave equation (\ref{eq: wave equation}) yields
\begin{equation}\label{eq: stationary wave equation}
 \begin{cases}
   z^2u(x) - \Delta u(x) = f & (x\in\Omega), \\
   \partial_n u(x) + z\hat{k}(z)u(x) = g &  (x\in\partial\Omega).
 \end{cases}
\end{equation}
Here $z$ is a complex number and formally $u=\hat{U}(z)=\int_0^{\infty} e^{-zt}U(t)dt$, $f=zU(0)+U_t(0)$ and $g=\hat{k}(z)U(0)|_{\partial\Omega}$. A way to give (\ref{eq: stationary wave equation}) a precise meaning is via the method of forms. Thus for $z\in\C\backslash(-\infty,0)$ let us define the bounded sesquilinear form $a_z:H^1\times H^1(\Omega)\rightarrow\C$ by
\begin{equation}\nonumber
 a_z(p, u) = z^2\int_{\Omega}p\overline{u} + \int_{\Omega} \nabla p \cdot\nabla\overline{u} + z\hat{k}(z)\int_{\partial\Omega} \Gamma p\Gamma \overline{u} dS. 
\end{equation}
If we replace the right-hand side $f,g$ by $F\in H^1(\Omega)^*$ (dual space of $H^1$), given by $\dual{F}{\eta}=\int_{\Omega}f\overline{\eta}+\int_{\partial\Omega} g\Gamma\overline{\eta}dS$, then a functional analytic realization of (\ref{eq: stationary wave equation}) is given by
\begin{equation}\label{eq: stationary wave equation FA}
 \forall \eta\in H^1(\Omega):\, a_z(u,\eta) = \dual{F}{\eta}_{(H^1)^*,H^1(\Omega)} .
\end{equation}
For all $z\in\C\backslash(-\infty,0)$ for which (\ref{eq: stationary wave equation FA}) has for all $F\in H^1(\Omega)^*$ a unique solution $u\in H^1(\Omega)$ we define the stationary resolvent operator $R(z):H^1(\Omega)^*\rightarrow H^1(\Omega), F\mapsto u$.
\begin{Theorem}[\cite{DFMP2010a}]\label{thm: spectrum DFMP}
 The spectrum of the wave operator satisfies
 \begin{align*}
  \sigma(-\A)\backslash(-\infty,0] &= \{z\in\C\backslash(-\infty,0]: R(z) \text{ does not exist.}\} \\
  &\subseteq \{z\in\C: \Re z < 0\}.
 \end{align*}
 Furthermore all spectral points in $\C\backslash(-\infty,0]$ are eigenvalues.
\end{Theorem}
Following the proof of the preceding theorem given by \cite{DFMP2010a} one sees that for $s\in\C\backslash i[0,\infty)$
\begin{equation}\label{A}
 (is+\A) (p, v, \psi) = (q, w, \varphi) \in \Ho
\end{equation}
is equivalent to
\begin{align}
 \label{R}
 \forall u\in H^1(\Omega): a_{is}(p, u) = \dual{F}{u}_{(H^1)^*,H^1(\Omega)} \\ \nonumber
  \text{ and } v = \frac{w + \nabla p}{is}, \, \psi(\tau) = \frac{\Gamma p + \varphi(\tau)}{is + \tau},
\end{align}
where
\begin{align}\nonumber
 \dual{F}{u} &= is\int_{\Omega} q \overline{u} - \int_{\Omega} w\cdot\nabla\overline{u} 
 - is\int_{\partial\Omega} \left[ \int_0^{\infty} \frac{\varphi(\tau)}{is+\tau} d\nu(\tau) \right] \Gamma u\, dS \\ \label{eq: F123}
 &=: \dual{F_1}{u} + \dual{F_2}{u} + \dual{F_3}{u} .
\end{align}
Observe that the adjoint operator of $R(z)$ is given by $R(z)^* = R(\overline{z})$ for all $z\in\C\backslash(-\infty,0)$ for which $R(z)$ is defined. Finally mention:
\begin{Theorem}[\cite{DFMP2010a}]\label{thm: A is injective}
 The wave operator $\A$ is injective.
\end{Theorem}
In the next section we characterize all kernels $k$ for which $\A$ is invertible.

\section{A correspondence between $(is+\A)^{-1}$ and $R(is)$}\label{sec: Correspondence between resolvents}
In this section we prove our first main result.
\begin{Theorem}\label{thm: Preliminary estimates}
The following holds:
 \begin{itemize}
 \item[(i)] Let $M:(0,\infty)\rightarrow[1,\infty)$ be an increasing function. Then
 \begin{align*}
  &\left[\exists s_1>0 \forall \abs{s}\geq s_1: \norm{(is+\A)^{-1}} \leq C M(\abs{s})\right] \\ \Leftrightarrow 
  &\left[\exists s_2>0 \forall \abs{s}\geq s_2: \norm{R(is)}_{L^2\rightarrow L^2} \leq C \abs{s}^{-1}M(\abs{s}) \right].
 \end{align*}
 \item[(ii)] $\exists s_3>0 \forall \abs{s}\leq s_3: \norm{(is+\A)^{-1}} \leq C \abs{s}^{-1}$.
 \item[(iii)] $\A$ is invertible iff $(\tau\mapsto\tau^{-1})\in L^{\infty}_{\nu}$, i.e. $\exists \varepsilon>0: \nu|_{(0,\varepsilon)}=0$.
 \end{itemize}
\end{Theorem}

If $\A$ is not invertible we deduce from Theorem \ref{thm: A is injective} that $\A$ can not be surjective in this case. In Section \ref{sec: range of A} we characterize the range of $\A$.

\subsection{Singularity at $\infty$}
In this subsection we prove Theorem \ref{thm: Preliminary estimates} (i). Therefore let us first define the auxiliary spaces $X^{\theta}$ by the real interpolation method:
\begin{equation}\nonumber
 X^{\theta} =
 \begin{cases}
  L^2(\Omega) \text{ resp. } H^1(\Omega) & \text{if } \theta = 0 \text{ resp. } 1, \\
  (L^2(\Omega), H^1(\Omega))_{\theta, 1} & \text{if } \theta\in(0,1), \\
  (X^{\theta})^* & \text{if } \theta\in[-1, 0).
 \end{cases}
\end{equation}
For $\theta\in(0,1)$ the space $X^{\theta}$ coincides with the Besov space $B^{\theta, 2}_1(\Omega)$. 

Let us explain why we use the Besov spaces $X^{\theta}$ instead of the Bessel potential spaces $H^{\theta}(\Omega)$. The reason is that while the trace operator $\Gamma: H^{\theta}(\Omega)\rightarrow H^{\theta-1/2}(\partial\Omega)$ is continuous for $\theta\in(1/2,1]$ this is no longer true for $\theta=1/2$ (with the convention $H^0=L^2$). On the other hand $\Gamma: X^{1/2}\rightarrow L^2(\partial\Omega)$ is indeed continuous (see Proposition \ref{thm: Trace} in the appendix). A corollary of this fact is that for some $C>0$
\begin{equation}\label{eq: Trace inequality}
 \forall u\in H^1(\Omega): \norm{\Gamma u}_{L^2(\partial\Omega)}^2 \leq C \norm{u}_{L^2(\Omega)}\norm{u}_{H^1(\Omega)}.
\end{equation}
Actually, by Lemma \ref{thm: interpolation lemma}, the preceding trace inequality is equivalent to the continuity of the trace operator $\Gamma:X^{1/2}\rightarrow L^2(\partial\Omega)$.

Let us prove the following extrapolation result.
\begin{Proposition}\label{thm: interpolation at infinity}
 Let $M:(1,\infty)\rightarrow[1,\infty)$ be an increasing function. If 
 \begin{equation}\label{eq: interpolation at infinity}
  \norm{R(is)}_{X^{-a}\rightarrow X^{b}} = O(\abs{s}^{a+b-1}M(\abs{s})) \text{ as } \abs{s}\rightarrow \infty
 \end{equation}
 is true for $a=b=0$, then it is also true for all $a,b\in[0,1]$.
\end{Proposition}
\begin{proof}
 Throughout the proof we may assume $\abs{s}$ to be sufficiently large. Assume that (\ref{eq: interpolation at infinity}) is true for $a = b = 0$. Let $f\in L^2(\Omega)$ and $p=R(is)f$, i.e. 
 \begin{equation}\nonumber 
 \forall u\in H^1(\Omega):\, a_{is}(p,u) = \int_{\Omega}f\overline{u} .
 \end{equation}
 Because of (\ref{eq: Trace inequality}) and the uniform boundedness of $\hat{k}(is)$ there are constants $c,C>0$ such that $\Re a_{is}(p,p)\geq c\norm{p}_{H^1}^2-Cs^2\norm{p}_{L^2}^2$. This helps us to estimate
 \begin{align*}
  c\norm{p}_{H^1}^2 &\leq \Re a_{is}(p,p) + Cs^2\norm{p}_{L^2}^2 \\
  &\leq \norm{f}_{L^2}\norm{p}_{L^2} + Cs^2\norm{p}_{L^2}^2 \\
  &\leq s^{-2}\norm{f}_{L^2}^2 + Cs^2\norm{p}_{L^2}^2 \\
  &\leq C M(\abs{s})^2 \norm{f}_{L^2}^2.
 \end{align*}
 In other words, (\ref{eq: interpolation at infinity}) is true for $a=0, b=1$. By duality (recall $R(z)^* = R(\overline{z})$) it is also true for $a=-1, b=0$. Almost the same calculation as above but now with the help of (\ref{eq: interpolation at infinity}) for the now known case $a=-1, b=0$ shows that (\ref{eq: interpolation at infinity}) is also true for $a=-1, b=1$. 
 
 It remains to interpolate. First interpolate between the parameters $(a=0,b=1)$ and $(a=1,b=1)$ to get (\ref{eq: interpolation at infinity}) for $a\in[0,1], b=1$. Then interpolate between the parameters $(a=0,b=0)$ and $(a=1,b=0)$ to get (\ref{eq: interpolation at infinity}) for $a\in[0,1], b=0$. One last interpolation gives us the desired result.
\end{proof}
Let us proceed with the proof of Theorem \ref{thm: Preliminary estimates} part (i). The implication ``$\Rightarrow$'' follows immediately from the equivalence of (\ref{A}) and (\ref{R}) with $w, \varphi = 0$. Therefore we have to show $\norm{\xnice}_{\Ho}\leq CM(\abs{s}) \norm{\ynice}_{\Ho}$, for all large $\abs{s}$ and for all $\xnice=(p,v,\psi)\in D(\A), \ynice=(q,w,\varphi)\in\Ho$ satisfying (\ref{A}) where $C$ does not depend on $s$ and $\ynice$.

Let $F_j$ for $j\in\{1,2,3\}$ be defined by (\ref{eq: F123}) and let $p_j$ satisfy
\begin{equation}\nonumber
 \forall u\in H^1(\Omega): a_{is}(p_j, u) = \dual{F_j}{u}_{(H^1)^*,H^1(\Omega)}.
\end{equation}
\emph{Case $j=1$.} It is clear that $\norm{F_1}_{L^2}=\abs{s}\norm{q}_{L^2}$. By Proposition \ref{thm: interpolation at infinity} we have $\norm{p_1}_{X^b}=O(\abs{s}^bM(\abs{s}))\norm{q}_{L^2}$ for all $b\in[0,1]$.
\emph{Case $j=2$.} It is clear that $\norm{F_2}_{X^{-1}}\leq\norm{w}_{L^2}$. By Proposition \ref{thm: interpolation at infinity} we have $\norm{p_2}_{X^b}=O(\abs{s}^bM(\abs{s}))\norm{w}_{L^2}$ for all $b\in[0,1]$.
\emph{Case $j=3$.} By the continuity of the trace $\Gamma:X^{1/2}\rightarrow L^2(\partial\Omega)$, H\"older's inequality and (\ref{eq: k is integrable}) we have
\begin{align*}
 \norm{F_3}_{X^{-\frac{1}{2}}} &\leq C \abs{s} \norm{\int_0^{\infty} \frac{\varphi(\tau)}{is+\tau} d\nu(\tau)}_{L^2(\partial\Omega)} \\
 &\leq C \abs{s}^{\frac{1}{2}} \norm{\varphi}_{L_{\nu}^2}.
\end{align*}
Again by Proposition \ref{thm: interpolation at infinity} this yields $\norm{p_3}_{X^b}=O(\abs{s}^bM(\abs{s}))\norm{\varphi}_{L_{\nu}^2}$ for all $b\in[0,1]$. Overall we derived the estimate $\norm{p}_{X^b}=O(\abs{s}^b M(\abs{s})) \norm{\ynice}_{\Ho}$ for all $b\in[0,1]$. Finally, this together with (\ref{R}) implies
\begin{align*}
 \norm{v}_{L^2} &\leq C\abs{s}^{-1}(\norm{w}_{L^2} + \norm{p}_{H^1}) \\
 &\leq CM(\abs{s})\norm{\ynice}_{\Ho}
\end{align*}
and
\begin{align*}
 \norm{\psi}_{L_{\nu}^2} &\leq \abs{s}^{-1}\norm{\varphi}_{L_{\nu}^2} 
   + \norm{\Gamma p}_{L^2}\left(\int_0^{\infty}\frac{1}{\abs{is+\tau}^2} d\nu(\tau)\right)^{\frac{1}{2}} \\
 &\leq \abs{s}^{-1}\norm{\varphi}_{L_{\nu}^2} + C\abs{s}^{-\frac{1}{2}}\norm{p}_{X^{\frac{1}{2}}} \\
 &\leq CM(\abs{s})\norm{\ynice}_{\Ho}.
\end{align*}
This concludes the proof of Theorem \ref{thm: Preliminary estimates} part (i).

\subsection{Singularity at $0$}
Now we prove Theorem \ref{thm: Preliminary estimates} (ii). For $s\neq 0$ we equip the Sobolev space $H^1(\Omega)$ with the equivalent norm $\norm{u}_{H_s^1}^2:=\norm{u}_{L^2}^2+\norm{s^{-1}\nabla u}_{L^2}^2$. In what follows we are interested in the asymptotics $s\rightarrow0$ while $s\neq0$. As in the preceding subsection we introduce some auxiliary spaces by the real interpolation method
\begin{equation}\nonumber
 X^{\theta}_s =
 \begin{cases}
  L^2(\Omega) \text{ resp. } H^1_s(\Omega) & \text{if } \theta = 0 \text{ resp. } 1, \\
  (L^2(\Omega), H^1_s(\Omega))_{\theta, 1} & \text{if } \theta\in(0,1), \\
  (X^{\theta}_s)^* & \text{if } \theta\in[-1, 0).
 \end{cases}
\end{equation}
We prove an analog of Proposition \ref{thm: interpolation at infinity} - but without the unknown function $M$.
\begin{Proposition}\label{thm: interpolation at 0}
 Let $a, b\in [0,1]$ and $\theta_+=\max\{a+b-1, 0\}$, then
 \begin{equation}\label{eq: interpolation at 0}
  \norm{R(is)}_{X^{-a}_s\rightarrow X^{b}_s} = O(\abs{s}^{-1-\theta_+}) \text{ as } s\rightarrow 0 .
 \end{equation}
\end{Proposition}
Before we can prove this proposition we show
\begin{Lemma}\label{thm: Maz'ya}
 There is a constant $C(\Omega)$ solely depending on the dimension and volume of $\Omega$ such that for all $u\in H^1(\Omega)$
 \begin{equation}\nonumber
  \int_{\Omega} \abs{\nabla u}^2 + \int_{\partial\Omega} \abs{u}^2 dS \geq C(\Omega) \int_{\Omega} \abs{u}^2 .
 \end{equation}
\end{Lemma}
\begin{proof}
 For the dimension $d=1$ this is an easy exercise for the reader. For $d\geq2$ we recall the isoperimetric inequality of Maz'ya \cite[Chapter 5.6]{Maz'ya} which is valid for all functions $v\in W^{1,1}(\Omega)$:
 \begin{equation}\nonumber
    \int_{\Omega} \abs{\nabla v} + \int_{\partial\Omega} \abs{v} dS 
    \geq \frac{d\sqrt{\pi}}{\Gamma(1+\frac{d}{2})^{\frac{1}{d}}} \left(\int_{\Omega} \abs{v}^{\frac{d}{d-1}}\right)^{\frac{d-1}{d}} .
 \end{equation}
 The right-hand side can easily be estimated from below by a constant times the $L^1(\Omega)$-norm of $v$ since $\Omega$ is bounded. The conclusion now follows by plugging in $v=u^2$.
\end{proof}
\begin{proof}[Proof of Proposition \ref{thm: interpolation at 0}.]
 Because of (\ref{eq: Trace inequality}) and the continuity of $\R\owns s\mapsto \hat{k}(is)$ at zero we have for all $u\in H^1(\Omega)$
 \begin{equation}\nonumber
  a_{is}(u,u) = \int_{\Omega} \abs{\nabla u}^2 + is\hat{k}(0)\int_{\partial\Omega} \abs{u}^2 dS + o(1)\norm{\nabla u}_{L^2}^2 + O(s^2)\norm{u}^2_{L^2} .
 \end{equation}
 Thus for sufficiently small $\abs{s}$ we deduce from Lemma \ref{thm: Maz'ya} and the fact $\hat{k}(0)>0$ that for all solutions $p\in H^1(\Omega)$ of the stationary wave equation (\ref{R}) with $F=f\in L^2(\Omega)$ the following estimate holds:
 \begin{align*}
  \abs{s}\norm{p}_{L^2}^2 &\leq C \abs{a_{is}(p,p)} = C \abs{\dual{f}{p}} \\
  &\leq C\abs{s}^{-1}\norm{f}^2_{L^2} + \frac{\abs{s}}{2}\norm{p}_{L^2}^2 .
 \end{align*}
 This shows (\ref{eq: interpolation at 0}) in the case $a=b=0$.
 
 Let us define the semi-linear functional
 \begin{equation}\nonumber
  G_s(u) = -s\int_{\Omega} \overline{u} + i\hat{k}(is)\int_{\partial\Omega} \overline{u} dS 
 \end{equation}
 for $u\in H^1(\Omega)$. Observe that $G_s(1)\rightarrow i\khat(0)\abs{\partial\Omega}\neq0$ as $s$ tends to $0$. It is easy to see from Poincar\'e's inequality (recall that $\Omega$ has Lipschitz boundary) that the expression $\norm{\nabla u}_{L^2} + \abs{G_s(u)}$ defines a norm on $H^1(\Omega)$ which is equivalent to the usual one - uniformly for small $\abs{s}$. In particular $p\mapsto\norm{\nabla p}_{L^2}$ is an equivalent norm on the kernel of $G_s$. 
 
 Remember that $p$ is the solution of (\ref{R}) for $F=f\in L^2(\Omega)$. We decompose $p=p_0+p_G$ with $p_G=G_s(p)=\text{const.}\in L^2(\Omega)$ and $G_s(p_0)=0$. Then
 \begin{equation}\nonumber
  a_{is}(p,p_0) = a_{is}(p_0,p_0) = (1+O(\abs{s}))\int_{\Omega} \abs{\nabla p_0}^2 .
 \end{equation}
 This implies 
 \begin{equation}\nonumber
  \norm{\nabla p_0}_{L^2}^2 \leq C\abs{a_{is}(p,p_0)} \leq C\abs{\dual{f}{p_0}} \leq C\norm{f}_{L^2}\norm{\nabla p_0} .
 \end{equation}
 This in combination with (\ref{eq: interpolation at 0}) for $a=b=0$ implies $\norm{p}_{H^1_s}\leq C\abs{s}^{-1}\norm{f}_{L^2}$ which is (\ref{eq: interpolation at 0}) for the parameters $a=0, b=1$. By duality (recall $R(z)^* = R(\overline{z})$) equation (\ref{eq: interpolation at 0}) is also true for $a=1,b=0$. A similar calculation as above with $f$ replaced by $F\in H^1(\Omega)^*$ and (\ref{eq: interpolation at 0}) for $a=1,b=0$ shows (\ref{eq: interpolation at 0}) for $a=1, b=1$. 
 
 What remains to do is some interpolation. It is important to interpolate in the right order. First, one has to show 
 \begin{equation}\nonumber
  \norm{R(is)}_{X^0_s\rightarrow X^{b_1}_s}, \norm{R(is)}_{X^{a_1}_s\rightarrow X^0_s}=O(\abs{s}^{-1}) 
 \end{equation}
 for $a_1,b_1\in [0,1]$. This can be done via interpolation between $(a=0,b=0)$ and $(a=0, b=1)$ for the first estimate and between $(a=0,b=0)$ and $(a=1, b=0)$ for the second estimate. Choosing $a_1$ and $b_1$ appropriately, the preceding estimates imply (\ref{eq: interpolation at 0}) in the case $a+b\leq 1$. Interpolation between the preceding case and $a=1,b=1$ yields the remaining part of the proposition.  
\end{proof}
Let us proceed with the proof of Theorem \ref{thm: Preliminary estimates} part (ii) in a similar fashion as for part (i). We have to show $\norm{x}_{\Ho}\leq C\abs{s}^{-1} \norm{y}_{\Ho}$ for all small $\abs{s}$ and for all $x=(p,v,\psi)\in D(A), y=(q,w,\varphi)\in\Ho$ satisfying (\ref{A}) where $C$ does not depend on $s$ and $y$. Let $F_j$ for $j\in\{1,2,3\}$ be defined by (\ref{eq: F123}) and let $p_j$ satisfy
\begin{equation}\nonumber
 \forall u\in H^1(\Omega): a_{is}(p_j, u) = \dual{F_j}{u}_{(H^1)^*,H^1(\Omega)}
\end{equation}
\emph{Case $j=1$.} It is clear that $\norm{F_1}_{L^2}=\abs{s}\norm{q}_{L^2}$. By Proposition \ref{thm: interpolation at 0} we have $\norm{p_1}_{X^b_s}=O(1)\norm{q}_{L^2}$ for all $b\in[0,1]$.
\emph{Case $j=2$.} It is clear that $\norm{F_2}_{X^{-1}_s}\leq \abs{s}\norm{w}_{L^2}$. By Proposition \ref{thm: interpolation at 0} we have $\norm{p_2}_{X^b_s}=O(\abs{s}^{-b})\norm{w}_{L^2}$ for all $b\in[0,1]$.
\emph{Case $j=3$.} By the continuity of the trace $\Gamma:X^{1/2}\rightarrow L^2(\partial\Omega)$ and by H\"older's inequality we have for all $\abs{s}\leq 1$
\begin{align*}
 \norm{F_3}_{X^{-\frac{1}{2}}_s} &\leq \norm{F_3}_{X^{-\frac{1}{2}}} \leq  
   C \abs{s} \norm{\int_0^{\infty} \frac{\varphi(\tau)}{is+\tau} d\nu(\tau)}_{L^2(\partial\Omega)} \\
 &\leq C \abs{s}^{\frac{1}{2}} \norm{\varphi}_{L_{\nu}^2}.
\end{align*}
By Proposition \ref{thm: interpolation at 0} this yields $\norm{p_3}_{X^b_s}=O(\abs{s}^{-\frac{1}{2}-(b-\frac{1}{2})_+})\norm{\varphi}_{L_{\nu}^2}$ for all $b\in[0,1]$. Overall we derived the estimate $\norm{p}_{X^b_s}=O(\abs{s}^{-\frac{1}{2}-(b-\frac{1}{2})_+}) \norm{\ynice}_{\Ho}$ for all $b\in[0,1]$. Finally, this together with (\ref{R}) implies
\begin{align*}
 \norm{v}_{L^2} &\leq C\abs{s}^{-1}(\norm{w}_{L^2} + \norm{\nabla p}_{L^2}) \\
 &\leq C\abs{s}^{-1}\norm{w}_{L^2} + C\norm{p}_{H^1_s} \\
 &\leq C\abs{s}^{-1}\norm{\ynice}_{\Ho}
\end{align*}
and because of $\norm{p}_{X^{\frac{1}{2}}}\leq\norm{p}_{X^{\frac{1}{2}}_s}$ for $\abs{s}\leq 1$
\begin{align*}
 \norm{\psi}_{L_{\nu}^2} &\leq \abs{s}^{-1}\norm{\varphi}_{L_{\nu}^2} 
   + \norm{\Gamma p}_{L^2}\left(\int_0^{\infty}\frac{1}{\abs{is+\tau}^2} d\nu(\tau)\right)^{\frac{1}{2}} \\
 &\leq \abs{s}^{-1}\norm{\varphi}_{L_{\nu}^2} + C\abs{s}^{-\frac{1}{2}}\norm{p}_{X^{\frac{1}{2}}_s} \\
 &\leq C\abs{s}^{-1}\norm{\ynice}_{\Ho}.
\end{align*}
This concludes the proof of Theorem \ref{thm: Preliminary estimates} part (ii).

\subsection{Spectrum at $0$}
Let us prove part (iii) of Theorem \ref{thm: Preliminary estimates}. 

``$\Rightarrow$''. Let us first assume that $\ynice=(q,w,\varphi)\in\Ho$ and $\xnice=(p,v,\psi)\in D(\A)$ satisfies (\ref{A}) for $s=0$. There is a function $u\in H^1(\Omega)$ such that $w=\nabla u$. We may assume $\int_{\Omega} u=0$ to make $u$ unique. Then (\ref{A}) for $s=0$ is
\begin{equation}\label{eq: A for s=0}
 \begin{cases}
  \Div v(x) = q(x) &  (x\in\Omega), \\
  \nabla p(x) = w(x) = \nabla u(x) & (x\in\Omega), \\
  \tau\psi(\tau,x)-p(x) = \varphi(\tau,x) & (\tau>0, x\in\partial\Omega), \\
  -v\cdot n(x) + \int_0^{\infty} \psi(\tau,x) d\nu(\tau) = 0 & (x\in\partial\Omega).
 \end{cases}
\end{equation}
From the second line we see that necessarily $p=u+\alpha$ for some complex number $\alpha$. We have
\begin{equation}\label{eq: psi}
 \psi = \frac{\varphi + \Gamma u + \alpha}{\tau} \in (L_{\nu}^1\cap L_{\nu}^2)(0,\infty; L^2(\partial\Omega)) .
\end{equation}
The $L^1_{\nu}$-inclusion follows by the definition of $D(\A)$ as explained in the paragraph following (\ref{eq: first order energy space}). Let us now specialize to the situation $q, w = 0$ and $\norm{\varphi}_{L^2_{\nu}}\leq 1$. Then $u=0$. By the existence of $\A^{-1}$ there must be a uniform bound $\abs{\alpha}\leq C$ where the constant does not depend on $\varphi$. Because of this, (\ref{eq: psi}) and $\int_0^{\infty}\tau^{-1} d\nu(\tau)<\infty$ we deduce a bound $\norm{\tau^{-1}\varphi}_{L^1_{\nu}}=\norm{\psi}_{L^1_{\nu}} + C\leq C$ where $C$ does not depend on $\varphi$. Since this is true for all $\varphi\in L^2_{\nu}(0,\infty; L^2(\partial\Omega))$ we deduce that the function $(0,\infty)\owns\tau\mapsto\tau^{-1}$ is in $L^2_{\nu}(0,\infty)$. If we use this in the $L^2_{\nu}$-inclusion in (\ref{eq: psi}) we see that $\norm{\tau^{-1}\varphi}_{L^2_{\nu}}=\norm{\psi}_{L^2_{\nu}} + C\leq C$ where $C$ does not depend on $\varphi$. Thus $\tau^{-1}$ is an $L^2_{\nu}$-multiplier and thus it must be bounded with respect to the measure $\nu$.

``$\Leftarrow$''. Assume now that $\nu|_{(0,\varepsilon)}=0$ for some $\varepsilon>0$. Given $\ynice=(q,w,\varphi)\in\Ho$ we show that there is a unique solution $\xnice=(p,v,\psi)\in D(\A)$ of (\ref{eq: A for s=0}). From the second line of (\ref{eq: A for s=0}) we see that necessarily $p=u+\alpha$ for some complex number $\alpha$ and $u$ as in the first part of the proof. The definition of $\Ho$ forces the necessity of the ansatz $v=-\nabla U$ for some function $U\in H^1(\Omega)$ with $\int_{\Omega} U=0$ for uniqueness purposes. It remains to uniquely determine $\alpha$ and $U$ since then $\psi$ is uniquely given by (\ref{eq: psi}). Let $h=-\int_0^{\infty} \psi d\nu\in L^2(\partial\Omega)$. Then the first and the last line of (\ref{eq: A for s=0}) are equivalent to
\begin{equation}\nonumber
 \begin{cases}
  -\Delta U(x) = q(x) &  (x\in\Omega), \\
  \partial_n U(x) = h(x) & (x\in\partial\Omega).
 \end{cases}
\end{equation}
By the Poincar\'e inequality this equation has a solution $U$ - which is unique under the constraint $\int_{\Omega} U=0$ - if and only if
\begin{align}\nonumber
 0 &= \int_{\Omega} q + \int_{\partial\Omega} h dS \\ \label{eq: condition on alpha}
 &= \int_{\Omega} q - \int_{\partial\Omega}\left( \khat(0) \Gamma u + \int_{\varepsilon}^{\infty} \frac{\varphi(\tau)}{\tau} d\nu(\tau)\right)dS -\alpha\abs{\partial\Omega}\khat(0). 
\end{align}
In the second equality we also used (\ref{eq: psi}). Since $\khat(0)\neq0$ this determines $\alpha$ and thus also $U$ uniquely. This completes the proof.

\subsection{The range of $\A$}\label{sec: range of A}
In the case that $\A$ is not invertible (i.e. $(\tau\mapsto\tau^{-1})\notin L^{\infty}_{\nu}$) in spite of Theorem \ref{thm: Martinez} it is important to know the image $R(\A)$ of $\A$. To characterize the range we have to distinguish two cases: (i) $(\tau\mapsto\tau^{-1})\in L^{2}_{\nu}$ and (ii) $(\tau\mapsto\tau^{-1})\notin L^{2}_{\nu}$. In case (ii) for a given $\varphi\in L^2_{\nu}(0,\infty;L^2(\partial\Omega))$ there might exist no $p\in H^1(\Omega)$ such that
\begin{equation}\nonumber
 \left( \tau \mapsto \frac{\varphi(\tau)+\Gamma p}{\tau} \right) \in L^2_{\nu}(0,\infty; L^2(\partial\Omega)) .
\end{equation}
In the case that $p$ exists, its boundary value $\Gamma p$ is uniquely determined and the function $(\tau\mapsto\varphi(\tau)/\tau)$ is integrable with respect to $\nu$. Therefore we can define the complex number
\begin{equation}\label{eq: mphi}
 m_{\varphi,p} = \int_{\partial\Omega} \int_0^{\infty} \frac{\varphi(\tau)+\Gamma p}{\tau} d\nu(\tau) dS .
\end{equation}
Equipped with this notation we can now formulate:

\begin{Theorem}\label{thm: range of A} Assume that $\A$ is not invertible (i.e. $(\tau\mapsto\tau^{-1})\notin L^{\infty}_{\nu}$).
 (i) If $(\tau\mapsto\tau^{-1})\in L^{2}_{\nu}$, then 
 \begin{equation}\nonumber
  R(\A) = \left\{ (q,w,\varphi)\in\Ho ; \int_0^{\infty} \norm{\frac{\varphi(\tau)}{\tau}}^2_{L^2(\partial\Omega)} d\nu(\tau) < \infty \right\} .
 \end{equation}
 (ii) If $(\tau\mapsto\tau^{-1})\notin L^{2}_{\nu}$, then 
 \begin{align*}
  R(\A) = \left\{ (q,w,\varphi)\in\Ho ; \exists p\in H^1(\Omega) :
           w=\nabla p, \int_{\Omega} q = m_{\varphi,p} \right. \text{ and }  \\
           \left. \int_0^{\infty} \norm{\frac{\varphi(\tau)+\Gamma p}{\tau}}^2_{L^2(\partial\Omega)} d\nu(\tau) < \infty \right\}
 \end{align*}
 where $m_{\varphi,p}$ is given by (\ref{eq: mphi}). If $(q,w,\varphi)$ is in the image of $\A$ then $p$ is unique. In fact it is the first component of the pre-image of $(q,w,\varphi)$.
\end{Theorem}

\begin{proof}
 Let $\ynice=(q,w,\varphi)\in\Ho$. Clearly $\ynice\in R(\A)$ if and only if we can find $\xnice=(p,v,\psi)\in\Ho_1$ such that $\A\xnice=\ynice$. Let $u\in H^1(\Omega)$ be such that $\nabla u = w$ and $\int_{\Omega} u = 0$. As in the proof of Theorem \ref{thm: Preliminary estimates}(iii) we see that necessarily $p=u+\alpha$ for some complex number $\alpha$ and 
 \begin{equation}\label{eq: range A definition of psi}
  \frac{\varphi + \Gamma p}{\tau} = \psi \in L_{\nu}^2(0,\infty; L^2(\partial\Omega)) .
 \end{equation}
 
 Let us assume that case (i) is valid. Then the so defined $\psi$ is in $L^2_{\nu}$ if and only if $(\tau\mapsto\varphi(\tau)/\tau)$ is square integrable with respect to $\nu$. Now one can proceed as in the ``$\Leftarrow$''-part of the proof of Theorem \ref{thm: Preliminary estimates}(iii) to find the unique $p$ and $v$ such that $\A \xnice = \ynice$.
 
 Let us now assume that case (ii) is valid. By (\ref{eq: range A definition of psi}) it is clear that the existence of $p$ as in the definition of $R(\A)$ is necessary. From the fact that $(\tau\mapsto\tau^{-1})$ is not square integrable we see that $\Gamma p$ is uniquely defined. Now we can again proceed as in the ``$\Leftarrow$''-part of the proof of Theorem \ref{thm: Preliminary estimates}(iii) to find the unique $p$ and $v$ such that $\A \xnice = \ynice$. The condition $\int_{\Omega} q = m_{\varphi,p}$ on $\ynice$ comes from (\ref{eq: condition on alpha}), where we have to replace $\khat(0) \Gamma u + \int_{\varepsilon}^{\infty} \frac{\varphi(\tau)}{\tau} d\nu(\tau) + \alpha\khat(0)$ by $\int_{0}^{\infty} \frac{\varphi(\tau)+\Gamma p}{\tau} d\nu(\tau)$ in our situation.
\end{proof}

\section{An upper estimate for $\norm{(is+\A)^{-1}}$ if $s\rightarrow\infty$}\label{sec: upper resolvent estimate}
We are seeking for an increasing function $M:[1,\infty)\rightarrow[1,\infty)$ such that for some constant $C>0$
\begin{align*}
 \norm{(is+\A)^{-1}} \leq C M(\abs{s}) \quad (\abs{s}\geq 1).
\end{align*}
In this section we want to show that the function $M(s) = 1/\Re \khat(is)$ is an upper bound (up to a constant) for the norm of $(is+\A)^ {-1}$ when $\abs{s}$ is large and if some additional assumptions on the acoustic impedance $\khat$ and the domain are satisfied.

More precisely we assume that the acoustic impedance satisfies
\begin{align}\label{eq: additional assumptions}
 \left[\abs{\khat}\frac{\abs{\khat}^2}{(\Re\khat)^2}\right](is) = o\left(\frac{1}{L(s)}\right) \text{ as } s\rightarrow\infty, \\ \nonumber
 \text{where } L(s) = s^{\alpha}(1+\log(s)) \text{ for } s \geq 1 .
\end{align}
The real number $\alpha\in[0,1)$ is a domain dependent constant which will be defined below. Note that for $\alpha\geq1$ there can not be any integrable completely monotonic function which satisfies this condition. 

Let $(u_j)$ be the sequence of normalized eigenfunctions of the Neumann Laplacian with respect to the corresponding (non-negative) frequencies $(\lambda_j)$. That is
\begin{equation}\label{eq: Neumann eigenvalue problem}
 \begin{cases}
   \lambda_j^2u_j(x) + \Delta u(x) = 0 & (x\in\Omega), \\
   \partial_n u_j(x) = 0 &  (x\in\partial\Omega), \\
   \norm{u_j}_{L^2(\Omega)} = 1 .
 \end{cases}
\end{equation}
The eigenfrequencies are counted with multiplicity and we may order them so that $0\leq \lambda_1 \leq \lambda_2 \leq \ldots$. We call a function $p\in L^2(\Omega)$ a spectral cluster of width $\delta>0$ whenever $\sup\{\abs{\lambda_j-\lambda_i}; a_j, a_i \neq 0\}\leq \delta$ where $p=\sum a_j u_j$ is the expansion of $p$ into eigenfunctions. We define the (mean) frequency $\lambda(p)\geq0$ of $p$ by $\lambda(p)^2=\sum \abs{(a_j/\norm{p}_{L^2})}^2\lambda_j^2$. We assume that the domain has the property that for sufficiently small $\delta>0$ there are constants $c,C>0$ such that for any spectral cluster $p$ of width $\delta$ the following estimate is true
\begin{align}\label{eq: upper and lower estimate for Neumann eigenfunctions}
 c\norm{p}_{L^2(\Omega)}^2 \leq \int_{\partial\Omega} \abs{\Gamma p}^2 dS \leq C \lambda(p)^{\alpha} \norm{p}_{L^2(\Omega)}^2 .
\end{align}
We call the left inequality the \emph{lower estimate} and the right inequality the \emph{upper estimate}. Note that the upper estimate is trivially satisfied for $\alpha=1$ by applying the trace inequality from Lemma \ref{thm: trace inequality}. It is indeed reasonable to assume that this estimate holds for some $\alpha$ strictly smaller than $1$. For example if the boundary of $\partial\Omega$ is of class $C^{\infty}$ then both estimates hold with $\alpha=2/3$. See \cite{BarnettHassellTacy2016} for this result. For $\Omega$ being an interval one can choose $\alpha=0$ and for a square $\alpha=1/2$ is optimal. 

This section is devoted to the proof of our second main result:

\begin{Theorem}\label{thm: upper resolvent estimate}
 Assume that (\ref{eq: additional assumptions}) is satisfied, where $\alpha\in[0,1)$ is such that (\ref{eq: upper and lower estimate for Neumann eigenfunctions}) holds for all spectral cluster $p$ of sufficiently small width $\delta>0$. Then there is a constant $C>0$ such that
 \begin{equation}\nonumber
  \norm{R(s)}_{L^2\rightarrow L^2} \leq \frac{C}{s\Re\khat(is)}
 \end{equation}
 for all $s\geq 1$.
\end{Theorem}
Compare this result to Theorem \ref{thm: Preliminary estimates} to obtain that the norm of $\norm{(is+\A)^{-1}}$ is bounded by $\frac{C}{\Re\khat(is)}$ under the constraints of the preceding theorem.

\subsection{Some auxiliary definitions}\label{sec: Some auxiliary definitions}
We fix a $\delta>0$ such that (\ref{eq: upper and lower estimate for Neumann eigenfunctions}) is true for any spectral cluster of width $3\delta$. For $p,q\in H^1(\Omega)$ we define the Neumann form by
\begin{align*}
 a_z^N(p,q) = z^2\int_{\Omega} p\overline{q} + \int_{\Omega} \nabla p\cdot \nabla \overline{q} .
\end{align*}
We cover $[0,\infty)$ by disjoint intervals $I_k=[k\lambda, (k+1)\lambda)$ for $k=0,1,2,\ldots$ such that
\begin{enumerate}
 \item [(i)] $\lambda \in [2\delta, 3\delta]$,
 \item [(ii)] $\exists k_c\in\N: I_{k_c} \supset (s-\delta, s+\delta)$.
\end{enumerate}
The covering depends on $s\geq1$ but this does not matter for our considerations. With the help of this partition we can uniquely expand every function $p\in L^2(\Omega)$ in terms of spectral clusters in the following way:
\begin{align*}
 p = \sum_{k=0}^{\infty} c_k p_k \text{ where }
 p_k = \sum_{\lambda_j\in I_k} a_j u_j, \, \norm{p_k}_{L^2(\Omega)}=1 .
\end{align*}
Let $s_k(p)\in I_k$ be such that
\begin{align*}
 s_k^2(p) = \int_{\Omega} \abs{\nabla p_k}^2.
\end{align*}
Let $p^0_{+(-)}=\sum_{k>(<)k_c} c_k p_k$ and $p^0=p^0_- + p^0_+$. Let $p_c=c_{k_c}p_{k_c}$. Obviously $p=p^0+p_c$. Define 
\begin{align*}
 p_+ =
  \begin{cases}
   p_+^0 + p_c & \text{ if } a_{is}^N(p_c)\geq0 , \\
   p_+^0 & \text{ else} ,
  \end{cases}
\end{align*}
and let $p_-$ be given by $p=p_+ + p_-$. Finally let $\ptilde=p_+-p_-$.

\subsection{Some auxiliary lemmas}
For the remaining part of Section \ref{sec: upper resolvent estimate} we use the notation introduced in Subsection \ref{sec: Some auxiliary definitions} and we assume that $\abs{s}\geq 1$. 
\begin{Lemma}\label{thm: auxiliary lemma 1}
 For all $p\in H^1(\Omega)$ we have $a_{is}^N(p,\ptilde)\geq \abs{s}\delta\norm{p^0}_{L^2(\Omega)}^2$.
\end{Lemma}
\begin{proof}
 \begin{align*}
  a^N_{is}(p,\ptilde) \geq a^N_{is}(p^0,(\ptilde)^0) 
  = \sum_{k\neq k_c} \abs{s^2 - s_k^2} \abs{c_k}^2 
  \geq s\delta \sum_{k\neq k_c} \abs{c_k}^2 
  = s\delta \norm{p^0}_{L^2(\Omega)}^2 .
 \end{align*} 
\end{proof}
A little bit more involved is the proof of the next lemma.
\begin{Lemma}\label{thm: auxiliary lemma 2}
 There is a constant $C>0$ (depending on $\delta$ and $\alpha$) such that for all $p\in H^1(\Omega)$
 \begin{equation}\nonumber
  \int_{\partial\Omega} \abs{\Gamma p^0}^2 dS \leq C \abs{s}^{\alpha}(1+\log(\abs{s})) \frac{a^N_{is}(p,\ptilde)}{\abs{s}}
 \end{equation}
\end{Lemma}
\begin{proof}
 Since $a^N_{is}(p,\ptilde)\geq a^N_{is}(p^0,(\ptilde)^0)$ we may assume that $p_c=0$. Because of 
 \begin{align*}
  \int_{\partial\Omega} \abs{\Gamma p}^2 dS \leq 2 \int_{\partial\Omega} \abs{\Gamma p_-}^2 + \abs{\Gamma p_+}^2 dS \\
  \text{and } a^N_{is}(p,\ptilde) = a^N_{is}(p_+) - a^N_{is}(p_-)
 \end{align*}
 we may assume without loss of generality that either $p=p_+$ or $p=p_-$. We show the proof in detail for the case $p=p_+$. The case $p=p_-$ is analogous and therefore we omit it. 
 \begin{align*}
  \int_{\partial\Omega} \abs{\Gamma p_+}^2 dS &= \norm{\sum_{k>k_c} c_k \Gamma p_k}_{L^2(\partial\Omega)}^2 \\
  &\leq \left(\sum_{k>k_c} \abs{c_k}\norm{\Gamma p_k}_{L^2(\partial\Omega)}\right)^2 \\
  &\leq \left(C\delta^{\frac{\alpha}{2}} \sum_{k>k_c} \abs{c_k} k^{\frac{\alpha}{2}}\right)^2 \\
  &\leq C\delta^{\alpha} \underbrace{\left(\sum_{k>k_c}\abs{c_k}^2(s_k^2-s^2)\right)}_{a^N_{is}(p_+)}
                         \underbrace{\left(\sum_{k>k_c}\frac{k^{\alpha}}{s_k^2 - s^2}\right)}_{=:J}
 \end{align*}
In the first line we used the continuity of the trace operator $\Gamma: H^1(\Omega)\rightarrow L^2(\partial\Omega)$. From the second to the third line we used the upper estimate (\ref{eq: upper and lower estimate for Neumann eigenfunctions}) together with $s_k\in I_k = \lambda[k,k+1)$ with $\lambda\in[2\delta,3\delta]$. It remains to estimate $J$. It is a well known trick to estimate sums of positive and decreasing summands by corresponding integrals.
\begin{align*}
 J &= \sum_{k>k_c} \frac{k^{\alpha}}{s_k^2 - s^2} \leq \sum_{k>k_c} \frac{k^{\alpha}}{\lambda^2k^2 - s^2} \\
 &\leq \frac{(k_c+1)^{\alpha}}{\lambda^2(k_c+1)^2-s^2} + \int_{k_c+1}^{\infty} \frac{x^{\alpha}}{\lambda^2x^2 - s^2} dx \\
 &=: J_1 + J_2 .
\end{align*}
It is not difficult to see that $J_1$ can be estimated by a constant times $\delta^{-1-\alpha}s^{\alpha-1}$. For $J_2$ we substitute $y=\lambda x/s$ and use that $\lambda(k_c+1)\geq1+\delta$. This yields
\begin{align*}
 J_2 &\leq C\delta^{-1-\alpha}s^{\alpha-1}\int_{1+\frac{\delta}{s}}^{\infty} \frac{y^{\alpha}}{y^2 - 1} dy \\
 &\leq C\delta^{-1-\alpha}s^{\alpha-1}\left(\int_{1+\frac{\delta}{s}}^{2} \frac{1}{y - 1} dy + \int_{2}^{\infty} \frac{1}{y^{2-\alpha}} dy\right) \\
 &\leq C \delta^{-1-\alpha}s^{\alpha-1}(\log(\frac{s}{\delta})+1) .
\end{align*}
This concludes the proof.
\end{proof}

\subsection{Proof of Theorem \ref{thm: upper resolvent estimate}}
Let $p\in H^1(\Omega)$ and $\abs{s}\geq 1$. We have to verify that 
\begin{equation}\nonumber
 \sup \{\abs{a_{is}(p,u)}; u\in H^1(\Omega),\, \norm{u}_{L^2(\Omega)}\leq 1\}\geq c\abs{s}\Re\khat(is)\norm{p}_{L^2(\Omega)}
\end{equation}
is true for some constant $c>0$ independent of $p$ and $s$. In the following we assume that $a^N_{is}(p_c)\geq 0$. This implies that $p_+=(p^0)_+ + p_c$ and $p_-=(p^0)_-$. The case $a^N_{is}(p_c)<0$ can be treated similarly and we therefore omit it. First we prove an auxiliary estimate with the help of Lemma \ref{thm: auxiliary lemma 2}:
\begin{align}\nonumber
 \int_{\partial\Omega} \abs{\Gamma p_+}^2 + \abs{\Gamma p_-}^2 dS &= \int_{\partial\Omega} \abs{\Gamma p_+^0}^2 + \abs{\Gamma p_-^0}^2 + \abs{\Gamma p_c}^2 + 2\Re (\Gamma p_+^0 \overline{\Gamma p_c}) dS \\
 \nonumber
 &\leq \int_{\partial\Omega} 2\abs{\Gamma p_+^0}^2 + \abs{\Gamma p_-^0}^2 + 2\abs{\Gamma p_c}^2 dS \\
 \label{eq: auxiliary estimate in main proof}
 &\leq CL(s)\frac{a^N_{is}(p,\ptilde)}{\abs{s}} + 2\int_{\partial\Omega} \abs{\Gamma p_c}^2 dS.
\end{align}
Let us define
\begin{equation}\nonumber
 L_1(s) = \left(\frac{\abs{\khat(is)}}{\Re\khat(is)}\right)^2 L(s) \geq L(s) .
\end{equation}
Our assumption (\ref{eq: additional assumptions}) on $k$ is equivalent to $|\khat|(is)=o(1/L_1(s))$ as $\abs{s}\rightarrow\infty$. Now we come to the final part of the proof which consists of distinguishing two cases. Essentially the first case means that $p$ is roughly the same as $p^0$ and the second case means that $p$ is roughly the same as $p_c$. We fix a constant $\varepsilon_1\in(0,1)$ to be chosen later. The choice of $\varepsilon_1$ does not depend on $s$.

\emph{Case 1: } $L_1(s)a^N_{is}(p,\ptilde)\geq \varepsilon_1 \abs{s}\int_{\partial\Omega}\abs{\Gamma p_c}^2 dS$. We first show that in this case the Neumann form dominates the form $a_{is}$ for $\abs{s}$ big enough in the following sense:
\begin{align*}
 \abs{a_{is}(p,\ptilde)-a^N_{is}(p,\ptilde)} &= \abs{s\khat(is)\int_{\partial\Omega}(\Gamma p_+ + \Gamma p_-)\overline{(\Gamma p_+ - \Gamma p_-)} dS}\\
 &\leq 2\abs{s\khat(is)} \int_{\partial\Omega} \abs{\Gamma p_+}^2 + \abs{\Gamma p_-}^2 dS \\
 &\leq C \abs{s\khat(is)} \varepsilon_1^{-1} L_1(s) \frac{a^N_{is}(p,\ptilde)}{\abs{s}} \\
 &\leq \frac{1}{2} a^N_{is}(p,\ptilde) .
\end{align*}
From the second to the third line we used the assumption of case 1 and (\ref{eq: auxiliary estimate in main proof}). By (\ref{eq: additional assumptions}) the last line is valid for all $s\geq s_0$, where $s_0$ is sufficiently large depending on how small $\varepsilon_1$ is. Therefore we have
\begin{align*}
 \abs{a_{is}(p,\ptilde)} &\geq (\frac{1}{4} + \frac{1}{4}) a^N_{is}(p,\ptilde) \\
 &\geq \frac{s\delta}{4}\norm{p^0}_{L^2(\Omega)}^2 + \frac{\varepsilon_1\abs{s}}{4L_1(s)}\int_{\partial\Omega}\abs{\Gamma p_c}^2 dS \\
 &\geq \frac{c\varepsilon_1 \abs{s}}{L_1(s)} \left(\norm{p^0}_{L^2(\Omega)}^2 + \norm{p_c}_{L^2(\Omega)}^2\right) \\
 &\geq c\varepsilon_1 \abs{s}\Re\khat(is)\norm{p}_{L^2(\Omega)}^2 .
\end{align*}
From the second to the third line we used the lower estimate (\ref{eq: upper and lower estimate for Neumann eigenfunctions}) and in the last step we used our assumptions on the acoustic impedance (\ref{eq: additional assumptions}). The theorem is proved for case 1.

\emph{Case 2: } $L_1(s)a^N_{is}(p,\ptilde) < \varepsilon_1 \abs{s}\int_{\partial\Omega}\abs{\Gamma p_c}^2 dS$. By Lemma \ref{thm: auxiliary lemma 1} and $\lim_{\abs{s}\rightarrow\infty} L_1(s)=\infty$ this yields
\begin{align}\label{eq: case 2 preliminary estimate}
 \int_{\partial\Omega}\abs{\Gamma p_c}^2 dS \geq \norm{p^0}_{L^2(\Omega)}^2
\end{align}
for all $\abs{s}\geq s_1$ with an $s_1>0$ not depending on $\varepsilon_1$. We show now that in case 2 the form $a_{is}$ is dominated by the contribution from the boundary. By Lemma \ref{thm: auxiliary lemma 2} we have 
\begin{align*}
 \abs{\int_{\partial\Omega} p^0\overline{p_c} dS} 
 &\leq \left(CL(s)\frac{a^N_{is}(p,\ptilde)}{\abs{s}}\right)^{\frac{1}{2}} \left(\int_{\partial\Omega} \abs{p_c}^2 dS \right)^{\frac{1}{2}} \\
 &\leq C\sqrt{\varepsilon_1} \left(\frac{L(s)}{L_1(s)}\right)^{\frac{1}{2}} \int_{\partial\Omega} \abs{p_c}^2 dS \\
 &\leq \frac{\Re\khat(is)}{2\abs{\khat(is)}}\int_{\partial\Omega} \abs{p_c}^2 dS .
\end{align*}
In the last step we choose $\varepsilon_1$ so small that $C\sqrt{\varepsilon_1}\leq1/2$. Finally from this, (\ref{eq: case 2 preliminary estimate}) and the lower estimate (\ref{eq: upper and lower estimate for Neumann eigenfunctions}) we deduce that
\begin{align*}
 \Im a_{is}(p, p_c) &\geq \frac{1}{2}\abs{s}\Re\khat(is)\int_{\partial\Omega} \abs{p_c}^2 dS \\
 &\geq c \abs{s}\Re\khat(is) (\norm{p_c}_{L^2(\Omega)}^2 + \norm{p^0}_{L^2(\Omega)}^2) \\
 &= c \abs{s}\Re\khat(is) \norm{p}_{L^2(\Omega)}^2
\end{align*}
which yields the claimed result.

\section{Examples}\label{sec: examples}
To illustrate our main results, Theorem \ref{thm: Preliminary estimates} and Theorem \ref{thm: upper resolvent estimate}, we want to consider special \emph{standard kernels} $k=k_{\beta,\varepsilon}$ (with $\varepsilon>0$ and $0<\beta<1$) introduced below. These standard kernels have the property that $\Re\khat(is)\approx |\khat(is)|\approx \abs{s}^{\beta-1}$ for large $\abs{s}$. This makes it easy to check whether (\ref{eq: additional assumptions}) is satisfied or not. We take a closer look at $\Omega$ being a square or a disk. In the case of the disk we show the optimality of the resolvent estimate, that is we show that $\norm{(is+\A)^{-1}}$ is not only bounded from above by a constant times $1/\Re\khat(is)$ but also from below. The standard kernels are designed in such a way that $\A$ is invertible (i.e. $(\tau\mapsto\tau^{-1})\in L^2_{\nu}$; see Theorem \ref{thm: Preliminary estimates}). We have assumed this for the simplicity of exposition. However, in Subsection \ref{sec: Example singularity at 0} we briefly show that our results yield (optimal) decay rates also in the presence of a singularity at zero.

The case $\Omega=(0,1)$ is treated separately in Section \ref{sec: 1D case}.

\subsection{Properties of the standard kernels}\label{sec: examples kernel}
For $\varepsilon>0$ and $0<\beta<1$ let
\begin{equation}\nonumber
 k_{\beta, \varepsilon}(t) = e^{-\varepsilon t}t^{-(1-\beta)} \text{ for } t>0 .
\end{equation}
To keep the notation short we fix $\varepsilon$ and $\beta$ now and write $k$ instead of $k_{\beta, \varepsilon}$ throughout this section. Obviously $k\in L^1(0,\infty)$ and for all $n\in\N_0$ we have $(-1)^n d^nk/dt^n(t) > 0$. The last property is a characterization of completely monotonic functions. Thus the kernel $k$ is admissible in the sense that the semigroup from Section \ref{sec: Semigroup approach} is defined. 

Let $\Gamma$ denote the Gamma function. Taking Laplace transform yields for $z>-\varepsilon$
\begin{align*}
 \khat(z) = \int_0^{\infty} e^{-(\varepsilon+z)t}t^{-(1-\beta)} dt 
 = \frac{1}{(\varepsilon+z)^{\beta}} \int_0^{\infty} s^{-(1-\beta)} e^{-s} ds 
 = \frac{\Gamma(\beta)}{(\varepsilon+z)^{\beta}} .
\end{align*}
By analyticity the equality between the left end and the right end of this chain of equations extends to $\C\backslash(-\infty,-\varepsilon]$. 

For $s\in\R$, let $\varphi(s)\in(-\frac{\pi}{2},\frac{\pi}{2})$ be the argument of $\varepsilon-is$. Note that $\varphi(s)\rightarrow\mp\frac{\pi}{2}$ as $s\rightarrow\pm\infty$. Then we have
\begin{equation}\nonumber
 \khat(is) = \Gamma(\beta)\abs{\frac{\varepsilon-is}{\varepsilon^2+s^2}}^{\beta}\left(\cos(\beta\varphi(s)) + i\sin(\beta\varphi(s))\right) .
\end{equation}
In particular
\begin{equation}\nonumber
 \Re\khat(is) \approx \abs{\Im\khat(is)} \approx \frac{1}{\abs{s}^{\beta}} \text{ for } \abs{s}\geq 1 .
\end{equation}
Here by $\approx$ we mean that the left-hand side is up to a constant, which does not depend on $s$, an upper bound for the right-hand side and vice versa. The first $\approx$-relation implies that the condition (\ref{eq: additional assumptions}) is equivalent to the simpler estimate $\Re\khat(is)=o(1/L(s))$ as $\abs{s}$ tends to infinity. More precisely we have
\begin{equation}\label{eq: standard additional assumptions}
 (\ref{eq: additional assumptions}) \Leftrightarrow \beta > \alpha .
\end{equation}

It is well known that for $z>0$ and $\beta\in(0,1)$
\begin{equation}\nonumber
 z^{-\beta} = \frac{\sin(\pi \beta)}{\pi} \int_0^{\infty} \frac{1}{\tau+z} \frac{d\tau}{\tau^{\beta}} .
\end{equation}
Thus
\begin{align*}
 \khat(z) = \frac{\sin(\beta\pi)}{\pi\Gamma(\beta)} \int_{\varepsilon}^{\infty} \frac{1}{\tau+z} \frac{d\tau}{(\tau-\varepsilon)^{\beta}} .
\end{align*}
In the notation of Section \ref{sec: Introduction} this means
\begin{equation}\nonumber
 d\nu(\tau) = \frac{\sin(\beta\pi)}{\pi\Gamma(\beta)}\cdot \frac{1_{[\varepsilon,\infty)}}{(\tau-\varepsilon)^{\beta}} d\tau .
\end{equation}
By Theorem \ref{thm: Preliminary estimates} (iii) we see that $\A$ is invertible.

\subsection{Smooth domains}
Let us suppose that $\Omega$ has a $C^{\infty}$ boundary and let $k=k_{\beta,\varepsilon}$ for some $\varepsilon>0$ and $0<\beta<1$. By \cite{BarnettHassellTacy2016} we know that (\ref{eq: upper and lower estimate for Neumann eigenfunctions}) is satisfied for $\alpha=2/3$. Thus by (\ref{eq: standard additional assumptions}) and Theorem \ref{thm: upper resolvent estimate} we have 
\begin{equation}
 \beta > \frac{2}{3} \quad \Longrightarrow \quad \forall s\in\R: \norm{(is+\A)^{-1}}\leq C (1+\abs{s})^{\beta}.
\end{equation}
By Theorem \ref{thm: Borichev-Tomilov} this implies
\begin{Proposition}\label{thm: arbitrary smooth domains}
 Let $\partial\Omega$ be of class $C^{\infty}$ and $k=k_{\beta,\varepsilon}$. If $\beta>2/3$ then, for all $t>0$ and $\xnice_0\in\Ho_1$,
 \begin{equation}\nonumber
  E(t, \xnice_0) \leq C t^{-\frac{2}{\beta}} E_1(\xnice_0).
 \end{equation}
\end{Proposition}

\subsection{The disk}\label{sec: examples disk}
Let $\Omega=D$ be the unit disk in $\R^2$. The smallest possible choice of $\alpha$ in (\ref{eq: upper and lower estimate for Neumann eigenfunctions}) is indeed $2/3$. The simple proof is based on a \emph{Rellich-type identity}, see for instance \cite[page 5]{BarnettHassellTacy2016}. So the circle already realizes the ``worsed case scenario'' with respect to the upper bounds for Neumann eigenfunctions. Thus in Proposition \ref{thm: arbitrary smooth domains} we cannot replace the condition $\beta>2/3$ by a weaker one. Instead we  show the optimality of the upper bound for the energy decay. Therefore we investigate the spectrum of $\A$.
\begin{Lemma}\label{thm: spectrum on disk}
 Let $\Omega=D$ and $k=k_{\beta,\varepsilon}$. Then there exists a sequence $(z_n)$ in the spectrum of $-\A$ such that $(\Im z_n)$ is positive and increasing and such that there exists a constant $C>0$ such that
 \begin{equation}\nonumber
   0 < -\Re z_n \leq \frac{C}{(\Im z_n)^{\beta}}
 \end{equation}
 holds for all $n$.
\end{Lemma}
As a corollary we have
\begin{equation}\nonumber
  \forall s>0: \sup_{\abs{\sigma}\leq s} \norm{(i\sigma+\A)^{-1}} \geq C (1+s)^{\beta}.
\end{equation}
By Theorem \ref{thm: Borichev-Tomilov} and the remark after Theorem \ref{thm: Batty-Duyckaerts} this implies
\begin{Proposition}
 Let $\Omega=D$ and $k=k_{\beta,\varepsilon}$. If $\beta>2/3$ then we have for all $t\geq1$ that
 \begin{equation}\nonumber
  c t^{-\frac{2}{\beta}} \leq \sup_{E_1(\xnice_0)\leq 1} E(t, \xnice_0) \leq C t^{-\frac{2}{\beta}} .
 \end{equation}
 If $\beta$ is arbitrary the left inequality remains valid.
\end{Proposition}

\begin{proof}[Proof of Lemma \ref{thm: spectrum on disk}]
 Except for the rate of convergence of $(z_n)$ towards the imaginary axis the content of our lemma is included in \cite[Theorem 5.2]{DFMP2010b}. Therefore we only sketch the existence of a sequence $(z_n)$ with imaginary part tending to infinity and real part tending to zero. 
 
 First recall that an eigenvalue is a complex number $z_n$ such that (\ref{eq: stationary wave equation FA}) with $F=0$ and $z=z_n$ has a non-zero solution $u$. After a transformation to polar coordinates, by a separation of variables argument one can show that the existence of $u$ is equivalent to the existence of a non-zero solution $v$ of 
 \begin{align*}
  \begin{cases}
   v''(r) + \frac{1}{r}v'(r) - (\frac{l^2}{r^2}+z^2)v(r) = 0 & (0 < r < 1), \\
   v'(1) + z\khat(z) v(1) = 0, & \\
   v(0+) \text{ is finite},
  \end{cases}
 \end{align*}
 for some $l\in\N_0$. The first and the third line forces that $v(r)$ is proportional to $J_l(izr)$, where $J_l$ is the $l$-th order Bessel function of the first kind (see e.g. \cite[Chapter 9]{AbramowitzStegun}). Therefore the second line implies
 \begin{equation}\label{eq: characteristic equation disk}
  \frac{J'_l(iz)}{J_l(iz)} = i\khat(z) .
 \end{equation}
 We have seen that a complex number $z_n\notin (-\infty, 0]$ is an eigenvalue of the wave operator if and only if it is a zero of (\ref{eq: characteristic equation disk}) for some $l$. Let us fix $l$ now. Following the approach of \cite{DFMP2010b} one can prove the existence of a sequence of zeros $(z_n)=(i s_n - \xi_n)$ with $s_n=n\pi + (2l+1)\pi/4$, $\Re\xi_n>0$ and $\xi_n$ tending to zero, by a Rouch\'e argument.
 
 It remains to prove that $\xi_n=O((\Im z_n)^{-(1-\beta)})$. By \cite[Formula 9.2.1]{AbramowitzStegun} the following asymptotic formula holds if $z$ tends to infinity while $\Re z$ stays bounded (and $l$ is fixed):
 \begin{equation}\label{eq: asymptotic expansion Bessel function}
  J_l(iz) = \sqrt{\frac{2}{\pi z}} \cos\left(iz - \frac{(2l+1)}{4}\pi\right) + O(\abs{z}^{-1}) .
 \end{equation}
 A naive way to get the corresponding asympotic formula for $J'$ and $J''$ would be to take derivatives of the cosine. In fact this yields the correct leading term. The error term is again $O(\abs{z}^{-1})$ in both cases. For the first derivative this is \cite[Formula 9.2.11]{AbramowitzStegun}. The formula for the second derivative then follows from the ordinary differential equation satisfied by $J_l$.
 
 Thus by a Taylor expansion of (\ref{eq: characteristic equation disk}) we get:
 \begin{equation}\nonumber
  0 + i\xi_n + O(\abs{\xi_n}^2+n^{-1}) = i\khat(is_n) - i\xi_n\khat'(is_n) + O(\abs{\xi_n}^2+n^{-1}) .
 \end{equation}
 This implies
 \begin{align}\label{eq: xi vs acoustic impedance}
  \xi_n &= (1+o(1))\khat(is_n) \\ \nonumber
  &= (1+o(1))\frac{\Gamma(\beta)}{s_n^{\beta}}\left(\cos(\beta\varphi(s_n)) + i\sin(\beta\varphi(s_n))\right).
 \end{align}
 Here $\varphi(s)$ is the argument of $\varepsilon-is$ (see Section \ref{sec: examples kernel}).
\end{proof}
Note that in the undamped case $k=0$ we have $z_n^{0}=s_n+O(s_n^{-1})$ by \cite[Formula 9.5.12]{AbramowitzStegun} for the eigenvalues $z_n^0$. Here again $s_n=n\pi + (2l+1)\pi/4$ and $l$ is fixed. Thus (\ref{eq: xi vs acoustic impedance}) implies that $z_n=z_n^0-(1+o(1))\khat(is_n)$.

\subsection{The square}
Let $\Omega=Q=(0,\pi)^2$ be a square. In terms of upper bounds for boundary values of spectral clusters the square behaves slightly better than the disk. It seems to be reasonable to believe that this is due to the fact that the square has no \emph{whispering gallery modes}.
\begin{Lemma}\label{thm: upper bound for square}
 Let $\Omega=Q$, $k=k_{\beta,\varepsilon}$ and $\delta>0$. If $\delta$ is sufficiently small then for each $L^2(Q)$-normalized spectral cluster $p$ of width $\delta$ of the Neumann-Laplace operator
 \begin{equation}\nonumber
  c \leq \int_{\partial\Omega} |\Gamma p|^2 dS \leq C s(p)^{\frac{1}{2}}.
 \end{equation}
 The constants $c,C>0$ do not depend on $p$. Furthermore the exponent $\alpha(Q)=1/2$ is optimal, i.e. one cannot replace it by a smaller one. 
\end{Lemma}
The optimality assertion of Lemma \ref{thm: upper bound for square} may be somewhat surprising. If $p$ was restricted to be a (pure) eigenfunction of the Neumann-Laplace operator the optimal exponent would be $\alpha=0$. This is a direct consequence of the explicit formula available for the eigenfunctions. However, it will be clear from the proof why spectral clusters behave differently.

As in the preceding examples the lemma implies
\begin{Proposition}\label{thm: square}
 Let $\Omega=Q$, $k=k_{\beta,\varepsilon}$. If $\beta>1/2$ then, for all $t>0$ and $\xnice_0\in\Ho_1$,
 \begin{equation}\nonumber
  E(t, \xnice_0) \leq C t^{-\frac{2}{\beta}} E_1(\xnice_0).
 \end{equation}
\end{Proposition}

\begin{proof}[Proof of Lemma \ref{thm: upper bound for square}]
 The explicit form of the normalized Neumann eigenfunctions $u_{m,n}$ and its eigenfrequencies $\lambda_{m,n}\geq0$ is
 \begin{align*}
  u_{m,n}(x,y) = 2\cos(mx)\cos(ny), \, \lambda_{m,n}^2 = m^2 + n^2.
 \end{align*}

 Let $p=\sum_{m,n} a_{n,m} u_{n,m}$ be a normalized spectral cluster of width $\delta$. We choose $s\geq0$ such that the set of indices $(m,n)$ with $a_{m,n}\neq 0$ is included in $I$ which is given by
 \begin{align*}
  I &= \{(m,n)\in\N_0^2; s^2\leq m^2 + n^2 \leq (s+\delta)^2\}, \\
  I_1 &= \{(m,n)\in I; m\geq n\}. 
 \end{align*}
 Without loss of generality we may assume that $\sum_{I_1}\abs{a_{m,n}}^2\geq 1/2$. We first prove the lower bound:
 \begin{align*}
  \int_{\partial\Omega} |\Gamma p|^2 dS  &= \sum_n \norm{\sum_m a_{m,n} \Gamma u_{m,n}}_{L^2(\partial\Omega)}^2 \\
  &\geq 16\pi\sum_{I_1} \abs{a_{m,n}}^2 \\
  &\geq 8\pi \norm{p}_{L^2(\Omega)}^2 .
 \end{align*}
In the first line we use the orthogonality relation for the cosine functions with respect to the $y$ variable. In the second line we use $\norm{u_{m,n}}_{L^2(\partial\Omega)}=4\sqrt{\pi}$ and the fact that the partial sum over $m$ in the preceding step includes only one member if $\delta$ is small and if the index set is restriced to $I_1$.

Let $N_n$ be the number of non-zero summands with respect to the inner sum in line one. It is not difficult to see that $N_n\leq C \sqrt{s}$ for a constant independent of $n$ and $s$. Therefore we have
 \begin{align*}
  \int_{\partial\Omega} |\Gamma p|^2 dS  &= \sum_n \norm{\sum_m a_{m,n} \Gamma u_{m,n}}_{L^2(\partial\Omega)}^2 \\
  &= \sum_n N_n^2 \norm{\frac{1}{N_n}\sum_m a_{m,n} \Gamma u_{m,n}}_{L^2(\partial\Omega)}^2 \\
  &\leq C \sum_{m,n} N_n \abs{a_{m,n}}^2 \\
  &\leq C s^{\frac{1}{2}} \norm{p}_{L^2(\Omega)}^2 .
 \end{align*}
 It remains to prove optimality of the exponent $\alpha=1/2$. For $n_1\in\N$ we consider a special spectral cluster $p_1$ of the form
 \begin{align*}
  p_1 = 2 \sum_{m=0}^{N-1} a_m \cos(m x)\cos(n_1 y) \\
  \text{where } N=N(n_1) = \lceil\varepsilon \sqrt{n_1}\rceil.
 \end{align*}
 If $\varepsilon>0$ is sufficiently small and $n_1$ large enough we see that $p_1$ is a spectral cluster of width $\delta$. If we set $a_m=1/\sqrt{N}$ we see that the $L^2(\Omega)$-norm of $p_1$ is $1$ and 
 \begin{align*}
  \int_{\{0\}\times (0,1)} \abs{\Gamma p_1}^2 dS = \abs{\sum_{m=1}^{N} a_m}^2 = N(n_1) \geq \varepsilon\sqrt{n_1} .
 \end{align*}
 This finishes the proof since $s(p_1)\in [n_1, n_1+\delta]$.
\end{proof}

\subsection{Decay in the presence of a singularity at zero}\label{sec: Example singularity at 0}
So far in this section we have excluded the case when $\A$ has a singularity at zero. The purpose of this subsection is to show that getting decay rates in this case is not more difficult than in the case where there is no singularity at zero. As in the previous subsection we simplify our presentation by considering a special family $(\khat'_{\alpha,\beta})_{\alpha,\beta}$ of acoustic impedances given by the measures
\begin{equation}\nonumber
   d\nu'_{\alpha,\beta} = \tau^{\alpha} d\tau|_{(0,1)} + (\tau-1)^{-\beta} d\tau|_{(1,\infty)} \quad (\alpha\in(0,\infty),\beta\in(0,1)).
  \end{equation}
  Obviously $(\tau\mapsto\tau^{-1})$ is integrable with respect to $\nu'_{\alpha,\beta}$ (thus $k'_{\alpha,\beta}$ is integrable) but it is not bounded with respect to that measure. Observe that $\alpha>1$ implies that $(\tau\mapsto\tau^{-1})$ is square integrable with respect to $\nu'$. In the following we assume for simplicity that $\alpha>1$. The reason is that by Theorem \ref{thm: range of A} the range of $\A$ has a simpler representation in this case.

  \begin{Lemma}\label{lem: visco kalphabeta}
   Let $\alpha\in(1,\infty),\beta\in(0,1)$. Then $(\tau\mapsto\tau^{-1})$ is integrable, square integrable but unbounded with respect to $\nu'_{\alpha,\beta}$. Moreover
   \begin{equation}\nonumber
    \khat'_{\alpha,\beta}(z) = \frac{\pi}{\sin(\beta\pi)} (1+z)^{-\beta} + O(\abs{z}^{-1})
   \end{equation}
   as $z$ tends to infinity avoiding $\R_-$.
  \end{Lemma}

  \begin{proof}
   We only have to prove the last statement. We calculate
   \begin{align*}
    \khat'(z) = \int_0^1 \frac{\tau^{\alpha}}{z+\tau} d\tau + \int_1^{\infty} \frac{1}{z+\tau} \frac{d\tau}{(\tau-1)^{\beta}} 
    =: I + II.
   \end{align*}
   It is easy to see that the modulus of $I$ is bounded by $(\abs{z}-1)^{-1}$ for all $z$ with $\abs{z} > 1$. With regard to $II$ we see that the well known identity 
   \begin{equation}\nonumber
    z^{-\beta} = \frac{\sin(\beta \pi)}{\pi} \int_0^{\infty} \frac{1}{z+\tau} \frac{d\tau}{\tau^{\beta}}
   \end{equation}
   finishes the proof.
  \end{proof}

  \begin{Proposition}\label{prop: visco singularity at infinity}
   Let $\alpha\in(0,\infty), \beta\in(2/3,1)$ and $k=k'_{\alpha,\beta}$. Let $\partial\Omega$ be a $C^\infty$-manifold. Then
   \begin{equation}\nonumber
    \norm{(is-\A)^{-1}} = O(\abs{s}^{\beta})
   \end{equation}
   as $\abs{s}>1$ tends to infinity.
  \end{Proposition}
  
  \begin{proof}
   This is an immediate consequence of Lemma \ref{lem: visco kalphabeta} together with Theorem \ref{thm: Preliminary estimates}(i) and Theorem \ref{thm: upper resolvent estimate}.
  \end{proof}
  
  We are now in the position to prove an \emph{optimal} decay estimate.
  
  \begin{Proposition}
   Let $\alpha\in(1,\infty), \beta\in(2/3,1)$ and $k=k_{\alpha,\beta}$. Let $\partial\Omega$ be a $C^\infty$-manifold. Then
   \begin{equation}\nonumber
    E(t,\xnice_0) \leq \frac{C}{1+t^2} \left[ E_1(\xnice_0)
                                                   + \int_0^{\infty} \norm{\psi_0(\tau)}_{L^2(\partial\Omega)}^2 \frac{d\nu(\tau)}{\tau^2} \right]
   \end{equation}
   holds for all $t\geq0$ and for all $\xnice_0=(p_0,v_0,\psi_0)\in\Ho$ for which the right-hand side is finite. The constant $C>0$ does not depend on $\xnice_0$ or $t$. Moreover this estimate is sharp in the sense that it would be invalid if one replaces $C/(1+t^2)$ by $o(1/(1+t^2))$ as $t$ tends to infinity.
  \end{Proposition}
  
  \begin{proof}
   Proposition \ref{prop: visco singularity at infinity}, Theorem \ref{thm: Preliminary estimates}(ii) together with \cite[Theorem 8.4]{BattyChillTomilov2016} yield
   \begin{equation}\nonumber
    \norm{e^{t\A}\xnice_0} \leq \frac{C}{1+t} \norm{\xnice_0}_{D(\A) \cap R(\A)} \text{ for all } \xnice_0\in D(\A) \cap R(\A).
   \end{equation}
   We know that the norm of $D(\A)$ is (equivalent to) the square root of the first order energy $E_1$. By Theorem \ref{thm: range of A} the norm on $R(\A)$ is given by
   \begin{equation}\nonumber
    \norm{\xnice_0}_{R(\A)}^2 = E(\xnice_0) + \int_0^{\infty} \norm{\psi_0(\tau)}_{L^2(\partial\Omega)}^2 \frac{d\nu(\tau)}{\tau^2}.
   \end{equation}
   This gives the desired estimate. The sharpness of this estimate follows from \cite[Theorem 6.9 and the remarks in Chapter 8]{BattyChillTomilov2016}.
  \end{proof}

\section{Optimal decay rates for the 1D case}\label{sec: 1D case}
Throughout this section $\Omega=(0,1)$ and $k$ is a completely monotonic, integrable function. We aim to show that in the 1D setting the conclusion of Theorem \ref{thm: upper resolvent estimate} remains true without any further hypothesis - like (\ref{eq: additional assumptions}) - on the acoustic impedance. Even more can be done - we prove that the upper estimate is optimal. More precisely we prove
\begin{Theorem}\label{thm: optimal resolvent estimate 1D}
 Let $\Omega=(0,1)$. Then there are constants $c,C>0$ such that for all $s>1$ we have
 \begin{equation}\nonumber
  \frac{c}{\Re\khat(is)} \leq \sup_{1\leq\abs{\sigma}\leq s} \norm{(i\sigma+\A)^{-1}} \leq \frac{C}{\Re\khat(is)}.
\end{equation}
\end{Theorem}
We prove the lower bound by investigating the spectrum of $-\A$ which is close to the imaginary axis (Subsection \ref{sec: the spectrum}). Furthermore we give a more or less concrete formula for the stationary resolvent operator $R(is)$ which allows to prove the upper bound (Subsection \ref{sec: Upper resolvent estimate}). Subsection \ref{sec: Decay rates} contains implications of Theorem \ref{thm: optimal resolvent estimate 1D} for the decay rates of the energy of the wave equation.

\subsection{The spectrum}\label{sec: the spectrum}
The spectrum of $-\A$ satisfies a characteristic equation which is implicitly contained in \cite{DFMP2010b}. For convenience of the reader we give a complete proof.
\begin{Proposition}
 A number $z\in \C\backslash(-\infty,0]$ is in the spectrum of $-\A$, and hence an eigenvalue, if and only if it satisfies
 \begin{equation}\label{eq: characteristic equation 1D}
  \left(\khat(z)-i\tan\left(\frac{iz}{2}\right)\right)\cdot \left(\khat(z)+i\cot\left(\frac{iz}{2}\right)\right) = 0
 \end{equation}
\end{Proposition}
\begin{proof}
 By Theorem \ref{thm: spectrum DFMP} together with the equivalence between (\ref{A}) and (\ref{R}) we see that $z$ is a spectral point if and only if there is a non-zero function $p$ solving
 \begin{align*}
  \begin{cases}
   z^2p(x) - p''(x) = 0 & (x\in(0,1)), \\
   -p'(0) + z\khat(z)p(0) = 0, &  \\
   p'(1) + z\khat(z)p(1) = 0. & 
  \end{cases}
 \end{align*}
 Up to a scalar factor the first two lines are equivalent to the following ansatz
 \begin{equation}\nonumber
  p(x) = \cos(iz x) - i\khat(z)\sin(iz x) .
 \end{equation}
 Plugging this into the third line yields that $z$ is an eigenvalue if and only if
 \begin{equation}\label{eq: precharacteristic equation}
  \left(\khat(z)^2 + 2i\khat(z)\cot(iz) + 1\right) z\sin(iz) = 0.
 \end{equation}
 Note that the zeros of the sine function do not lead to an eigenvalue since the cotangent function has a singularity at the same point. Actually we already know from the situation of general domains that an eigenvalue which is neither zero nor a negative number must have negative real-part. Thus we may simplify (\ref{eq: precharacteristic equation}) by dividing by $z\sin(iz)$. The claim now follows from the formula $\cot(\zeta)-\tan(\zeta)=2\cot(2\zeta)$ which is valid for all complex numbers $\zeta$.
\end{proof}

Let $H,R>0$. The reader may consider $H$ and $R$ as large numbers. We are interested in the part of the spectrum of $-\A$ contained in the strip
\begin{equation}\nonumber
 \Ueps = \{z\in\C; \abs{\Im z}>R \text{ and } 0 < -\Re z < H\} .
\end{equation}
\begin{Proposition}\label{thm: spectrum 1D}
 Let $H>0$. Then for $R>0$ large enough there exists a natural number $n_0>0$ such that the part of the spectrum of $-\A$ which is contained in $\Ueps$ is given by a doubly infinite sequence $(z_n)_{n=\pm n_0}^{\infty}$ with $z_{-n}=\overline{z_n}$ for all $n$ and
 \begin{align*}
  \Im z_n &= \pi n - \left[(2+O(|\khat|))\Im\khat\right](i\pi n), \\
  \Re z_n &= -\left[(2+O(|\khat|))\Re\khat\right](i\pi n).
 \end{align*}
\end{Proposition}
As a consequence the lower bound in Theorem \ref{thm: optimal resolvent estimate 1D} is proved.

Note that the two asymptotic formulas given by the proposition imply $z_n=(2+o(1))\khat(in\pi)$ for $n$ tending to plus or minus infinity. This formula can be proved by the same Taylor expansion argument as in the proof of Lemma \ref{thm: spectrum on disk}. See also the remark after the proof of the mentioned lemma. But this is not enough in order to prove the lower bound in Theorem \ref{thm: optimal resolvent estimate 1D} since it might happen that the real part of $\khat(is)$ tends much faster to zero then its imaginary part! This explains the more elaborate Taylor expansion technique in the proof below.

\begin{proof}[Proof of Proposition \ref{thm: spectrum 1D}]
 We are searching for the solutions $z\in\Ueps$ of the characteristic equation (\ref{eq: characteristic equation 1D}). For simplicity we only consider the solutions of
 \begin{equation}\nonumber
  z\in\Ueps \text{ and } F(z) := \khat(z) - i\tan\left(\frac{iz}{2}\right) = 0.
 \end{equation}
 We apply a Rouch\'e argument to show that the zeros of this equation are close to the zeros $is_{2n}=2n\pi i$ of the tangens-type function on the right-hand side. Let $(\varepsilon_{2n})$ be a null-sequence of positive real numbers, smaller than $H$, to be fixed later. Let $B_{2n}$ be the open ball of radius $\varepsilon_{2n}$ around the center $is_{2n}$. For $r>0$ let 
 \begin{equation}
  \Veps(r)=\{z\in\C; R < \Im z < R+r \text{ and } -H < \Re z < H\}. 
 \end{equation}
 Take $K(r)$ to be the boundary of the set $\Veps(r)\backslash (\bigcup_n B_{2n})$. Since $\khat(z)$ tends to zero as $z$ tends to infinity with bounded real part we can choose $R$ so large and $(\varepsilon_{2n})$ so slowly decreasing such that $|\khat(z)|<\abs{i\tan(iz/2)}$ for $z\in K(r)$. Thus Rouch\'e's theorem for meromorphic functions says that for $F$ and for $(z\mapsto i\tan(iz/2))$ restricted to $\Veps(r)$ the number of zeros minus the number of poles (counted with multiplicity) is the same for all $r>0$. The poles of $F$ are actually the same as for for the tangens type function. Thus it is proved that for large enough $n_0\in\N$ the zeros of $F$ from $\Ueps$ for $R=(2n_0-1)\pi$ are simple and contained in the balls $B_{2n}$ for $\abs{n}\geq n_0$. Note that we used that we already know that zeros of the characteristic equation must have negative real part.
 
 We have verified that all zeros $z_{2n}$ of $F|_\Ueps$ are given by the following ansatz:
 \begin{equation}\nonumber
  z_{2n} = is_{2n} - \xi_{2n} \text{ with } \Re\xi_{2n}>0 \text{ and } \xi_{2n} = o(1).
 \end{equation}
 In the remaining part of the proof we want to simplify the notation by dropping the indices from $z, s$ and $\xi$. We also write $\khat$ instead of $\khat(z)$. It is not difficult to verify that $F(z)=0$ is equivalent to 
 \begin{equation}\nonumber
  e^z = \frac{1-\khat}{1+\khat} = \frac{(1-i\Im\khat) - \Re\khat}{(1+i\Im\khat) + \Re\khat}.
 \end{equation}
 Let $\alpha=\arg(1+i\Im\khat)$ be the argument of $1+i\Im\khat$ and $L=(1+(\Im\khat)^2)^{1/2}$. Then
 \begin{align*}
  &\arg(1\pm\khat) = \pm \alpha (1+O(\Re\khat)) = \pm (1+O(|\khat|))\Im\khat, \\
  &\text{thus } \Im \xi = 2(1+O(|\khat|))\Im\khat.
 \end{align*}
 This yields the first asymptotic formula claimed in the proposition. The second asympotic formula is a direct consequence of
 \begin{equation}\nonumber
  e^{-\Re\xi} = \frac{L-(1+O(\abs{\alpha}^2))\Re\khat}{L+(1+O(\abs{\alpha}^2))\Re\khat} = 1 - \frac{2}{L}(1+O(|\Im\khat|^2))\Re\khat + O((\Re\khat)^2).
 \end{equation}
\end{proof}

\subsection{Upper resolvent estimate}\label{sec: Upper resolvent estimate}
We prove the upper estimate stated in Theorem \ref{thm: optimal resolvent estimate 1D}. By Theorem \ref{thm: Preliminary estimates} part (i) it suffices to show
\begin{Proposition}
 For all $\abs{s}\geq 1$ we have $\norm{R(is)}_{L^2\rightarrow L^2}\leq C(\abs{s}\Re\khat(is))^{-1}$.
\end{Proposition}
\begin{proof}
 For some $f\in L^2(0,1)$ let $p$ be the solution of
  \begin{align}\label{eq: stationary equation 1D}
  \begin{cases}
   -s^2p(x) - p''(x) = f & (x\in(0,1)), \\
   -p'(0) + is\khat(is)p(0) = 0, &  \\
   p'(1) + is\khat(is)p(1) = 0. & 
  \end{cases}
 \end{align}

Let us define two auxiliary functions $p_f$ and $p^0$ by
\begin{align*}
 p_f(x) = -\frac{1}{s}\int_0^x \sin(s(x-y))f(y) dy \text{ and } p^0(x) = \cos(sx) + i\khat(is)\sin(sx).
\end{align*}
It is easy to see that $p=ap^0+p_f$ with $a\in\C$ is the only possible ansatz which satisfies the first two lines in (\ref{eq: stationary equation 1D}). The parameter $a$ is uniquely defined by the condition from the third line. A short calculation yields that this condition is equivalent to
\begin{equation}\nonumber
 a s\cdot \underbrace{\left( \khat(is) + i\tan\left(\frac{s}{2}\right) \right) \left( \khat(is) - i\cot\left(\frac{s}{2}\right) \right)}_{=:D(s)}\cdot \sin(s)
 = - p_f'(1) - is\khat(is) p_f(1).
\end{equation}
Note that the singularities of $D$ cancel the zeros of the sine function. Thus we have an explicit formula for $a$ in terms of $f$. Further note that the absolute values of $sp_f(1)$ and $p_f'(1)$ can be estimated from above by a constant times $\norm{f}_{L^2}$. Thus 
\begin{equation}\nonumber
 \abs{a} \leq \frac{C}{\abs{s}} \cdot \frac{1}{\abs{D(s)\sin(s)}}\cdot \norm{f}_{L^2(0,1)}.
\end{equation}
By the presence of the tangent and contangent type function the factor $D(s)\sin(s)$ can only be small in a neighbourhood of $s=2n\pi$ or $s=(2n+1)\pi$. But in this case the real part of $\khat$ prevents $D$ from getting too small. We thus have an estimate $\abs{D(s)\sin(s)}\geq c \Re\khat(is)$ for $\abs{s}\geq 1$ which in turn gives an upper bound on $\abs{a}$. Since the $L^2$-norm of $p^0$ can be estimated from above by a constant the proof is finished.
\end{proof}

\subsection{Decay rates}\label{sec: Decay rates}
An immediate consequence of Theorem \ref{thm: optimal resolvent estimate 1D}, \ref{thm: Preliminary estimates}(iii), \ref{thm: Batty-Duyckaerts} and the remark after Theorem \ref{thm: Batty-Duyckaerts} is
\begin{Theorem}\label{thm: Decay rate 1D}
 Assume that $\nu|_{(0,\varepsilon)}=0$ for some $\varepsilon>0$. Then there are constants $c,C>0$ and $t_0>0$ such that for all $t\geq t_0$
 \begin{equation}\nonumber
  c M^{-1}(t/c) \leq \sup_{E_1(\xnice_0)\leq 1} E(t, \xnice_0) \leq C M_{log}^{-1}(t/C)
 \end{equation}
 where the increasing function $M:(0,\infty)\rightarrow(0,\infty)$ is given by $M(s)=(\Re\khat(is))^{-1}$.
\end{Theorem}
We made the assumption $\nu|_{(0,\varepsilon)}=0$ (i.e. $\A$ is invertible) only to simplify the formulation of the theorem. A recipe how to adapt the formulation in case of a non-invertible $\A$ is given in Subsection \ref{sec: Example singularity at 0}.

\section{Further research}\label{sec: Conclusion}
For a complete treatment of resolvent estimates for wave equations like (\ref{eq: wave equation}) it would be desirable to answer at least the following two questions:

\textbf{Question 1.} Is the upper bound on $\norm{(is-\A)^{-1}}$, given by Theorem \ref{thm: upper resolvent estimate}, optimal? 

\textbf{Question 2.} Can one discard the additional assumption (\ref{eq: additional assumptions}) on $\khat$ without changing the conclusion of Theorem \ref{thm: upper resolvent estimate}?

A strategy to positively answer question 1 is to show that there exists a sequence of eigenvalues of $-\A$ which tend to infinity and approach the imaginary axis sufficiently fast. We have seen that this strategy works at least for $\Omega=(0,1)$ and $\Omega=D$ (see Section \ref{sec: 1D case} and Subsection \ref{sec: examples disk}). For the disk we restricted to kernels $k=k_{\beta,\varepsilon}$. However, with the more elaborate Taylor argument which proved Proposition \ref{thm: spectrum 1D} one can discard this restriction from Lemma \ref{thm: spectrum on disk}. We believe that there is a general argument for any bounded Lipschitz domain $\Omega$ yielding the existence of such a sequence of eigenvalues.

By our investigations in Section \ref{sec: 1D case} we already have a positive answer for question 2 in the 1D setting. Moreover, if $\Omega=D$ is the disk we already know from the spectrum that an increasing function $M$ with $M(s)=o((\Re\khat(is))^{-1})$ can never be an upper bound for $\norm{(is+\A)^{-1}}$ for all large $\abs{s}$. We think that the answer to question 2 is either ``yes'', or if ``no'' then the upper bound solely depends on $\Re\khat$ and the infimum of all $\alpha$ making the upper estimate in (\ref{eq: upper and lower estimate for Neumann eigenfunctions}) true for all spectral cluster $p$. 

Concerning the application of resolvent estimates to energy decay there is also a third question. Let us assume for a moment that the answers to questions 1 and 2 were positive. Then Theorem \ref{thm: Decay rate 1D} was true for any $\Omega$. In general it is not possible to replace $M_{log}$ by $M$ in Theorem \ref{thm: Batty-Duyckaerts}. However, does our particular situation allow for a smaller upper bound? In our opinion the most elegant result would be a positive answer to 

\textbf{Question 3.} Is Theorem \ref{thm: Decay rate 1D} true for all bounded Lipschitz domains $\Omega$ - even with $M_{log}$ replaced by $M$?

\begin{appendix}
\section{Besov spaces and the trace operator}\label{apx: Traces}

In this article we work with fractional Sobolev spaces, Besov spaces and the trace operator acting on them. Note also that we work with the space $H^s(\partial\Omega)$ which is not only a fractional Sobolev space but also is a function space on a closed subset of $\R^d$ which has non-empty interior. In this appendix we aim at providing some results from the literature about Sobolev/Besov spaces and their relation to interpolation spaces which is necessary to follow the arguments from our article.

Of exceptional importance for the proof of Theorem \ref{thm: Preliminary estimates} (i) and (ii) is the validity of the borderline trace theorem - Proposition \ref{thm: Trace}. This borderline case seems to be well-known to the experts - also for Lipschitz domains - but unfortunately we were not able to find it in the literature except in \cite[Theorem 18.6]{Triebel}. But the proof given there is not in our spirit - the Besov spaces are not defined as interpolation spaces there. Therefore we give a simple direct proof via the characterization of Besov spaces as interpolation spaces which is true if $\Omega$ has the so called extension property.

\subsection{Fractional Sobolev- and Besov spaces}\label{apx: fractional Sobolev and Besov spaces}
Let $\Omega\subset\R^d$ be a bounded Lipschitz domain. Here by \emph{Lipschitz} we mean that locally near any boundary point and in an appropriate coordinate system one can describe $\Omega$ as the set of points which are above the graph of some Lipschitz continuous function from $\R^{d-1}$ into $\R$.

Let $1\leq p\leq \infty$. We assume the reader to be familiar with the usual \emph{Sobolev space} $W^{1,p}(\Omega)$ which consists of all functions $u\in L^p(\Omega)$ for which all distributional derivatives $\partial_ju$ are in $L^p(\Omega)$. There are different methods of defining Besov spaces. For our purposes it is most convenient to define the \emph{Besov spaces} for $0<s<1$ and $1\leq q\leq \infty$ as real interpolation spaces:
\begin{equation}\
 B^{s,p}_q(\Omega) = (L^p(\Omega), W^{1,p}(\Omega))_{s, q} .
\end{equation}
Another approach is to define $B^{s,p}_q(\R^d)$ for example via interpolation and then to define the Besov space on $\Omega$ as restrictions to $\Omega$ of Besov function on $\R^d$. In general these approaches are not equivalent but if $\Omega$ satisfies the extension property they are equivalent \cite[Chapter 34]{Tartar}. In our setting ($0<s<1$) we say that $\Omega$ satisfies the \emph{extension property} if there is a linear and continuous operator $\Ext:W^{1,p}(\Omega)\rightarrow W^{1,p}(\R^d)$ such that $(\Ext u)|_{\Omega}=u$ for each $u$ from $W^{1,p}(\Omega)$. The extension property is fulfilled if $\Omega$ is bounded and has a Lipschitz boundary. In the following we always assume that this extension property is fulfilled - otherwise some statements from below are not valid.

The \emph{Sobolev-Slobodeckij spaces} are defined as special Besov spaces $W^{s,p}(\Omega)=B^{s,p}_p(\Omega)$. It is common to write $H^s$ instead of $W^{s,2}$ in the Hilbert space setting. For $0\leq s \leq 1$ it is also possible to define the scale of \emph{fractional Sobolev spaces} (also known as \emph{Bessel potential spaces}) $H^{s,p}(\Omega)$ via Fourier methods for the special case $\Omega=\R^d$ and via restriction for the general case. These spaces form a scale of complex interpolation spaces. In general the fractional Sobolev spaces differ from the Sobolev-Slobodeckij spaces but coincide in the case $p=2$ (see \cite[Chapter 7.67]{AdamsFournier}. Note that in Adam's and Fournier's book the letter $W$ stands for the fractional Sobolev spaces. We also have $H^{1,p}(\Omega)=W^{1,p}(\Omega)$ for $1<p<\infty$ - which is Calder\'on's Theorem (see \cite[page 7]{JonssonWallin}).

We mention that for all $0<s_1\leq s<1$ and $q,q_1\in [1,\infty]$ with the restriction $q\leq q_1$ if $s_1=s$:
\begin{align*}
 B^{s, p}_q(\Omega) \hookrightarrow B^{s_1, p}_{q_1}(\Omega) .
\end{align*}
This is a direct consequence of a general result about the real interpolation method (see e.g. \cite[Lemma 22.2]{Tartar}).

It is possible to define the Besov space $B^{s,p}_q(A)$ on a general class of closed subsets $A$ of $\R^d$ - the so called \emph{$d$-sets}. For $\Omega$ having a Lipschitz boundary its boundary $\partial\Omega$ is such a set, since it is a $d-1$ dimensional manifold topologically. The required background is included in \cite[Chapter V]{JonssonWallin}. Again we write $H^s(A)=B^{s, 2}_2(A)$ in the Hilbert space setting.

\subsection{Traces for functions with $1/p$ or more derivatives}
Throughout this subsection $\Omega\subseteq\R^d$ is a bounded domain with Lipschitz boundary and we let $1<p<\infty$. For $1/p<s<1$ the following Theorem is a special case of \cite[Chapter VI, Theorem 1-3]{JonssonWallin}. For $s=1$ it is a special case of \cite[Chapter VII, Theorem 1-3]{JonssonWallin}, keeping in mind that by Calder\'on's Theorem the Bessel potential spaces are the ordinary Sobolev spaces for positive integer orders $s$.
\begin{Theorem}
 Let $1/p<s<1$. Then the trace operator $\Gamma: C(\overline{\Omega})\rightarrow C(\partial\Omega), u\mapsto u|_{\partial\Omega}$ extends continuously to an operator
 \begin{equation}\nonumber
  \Gamma:B^{s,p}_{q}(\Omega)\rightarrow B^{s-\frac{1}{p}, p}_q(\partial\Omega). 
 \end{equation}
Furthermore $\Gamma$ has a continuous right inverse:
 \begin{equation}\nonumber
  \Ext:B^{s-\frac{1}{p}, p}_q(\partial\Omega)\rightarrow B^{s,p}_{q}(\Omega),\quad \Gamma\circ\Ext = \id_{B^{s-\frac{1}{p}, p}_q(\partial\Omega)} .
 \end{equation}
 The theorem remains valid for $s=1$, $q=p$ if one replaces $B^{s,p}_q(\Omega)$ by $W^{1,p}(\Omega)$.
\end{Theorem}

Unfortunately this theorem is false for any $1\leq q\leq \infty$ in the borderline case $s=1/p$ if one replaces the target space of $\Gamma$ by $L^p(\partial\Omega)$. But for our purposes it is sufficient that a weakened version remains valid.

\begin{Proposition}\label{thm: Trace}
 The trace operator $\Gamma: B^{\frac{1}{p},p}_1(\Omega)\rightarrow L^p(\partial\Omega)$ is continuous.
\end{Proposition}
Actually $\Gamma$ from this proposition is indeed surjective (but we do not need this property in our article) and a more general version is proved in \cite[Section 18.6]{Triebel}. However there is no \emph{linear} extension operator from $L^p(\partial\Omega)$ back to the Besov space (See \cite{Triebel} and references therein).

We indicate a simple direct proof of Proposition \ref{thm: Trace}. It is based on two lemmas which have very simple proofs on their own. The first one is
\begin{Lemma}\label{thm: trace inequality}
 For every $C^{\infty}$ function $u$ with compact support in $\R^d$ it is true that
 \begin{equation}\nonumber
  \norm{\Gamma u}_{L^p(\partial\Omega)} \leq C\norm{u}_{L^p(\Omega)}^{1-\frac{1}{p}}\norm{u}_{W^{1,p}(\Omega)}^{\frac{1}{p}}.
 \end{equation}
\end{Lemma}
The straightforward proof can be found in \cite[Lemma 13.1]{Tartar}. For a different proof in the case $p=2$ we refer to \cite{Monniaux2014}. The second ingredient to the proof of Proposition \ref{thm: Trace} is \cite[Lemma 25.3]{Tartar} which we recall here for convenience of the reader.
\begin{Lemma}\label{thm: interpolation lemma}
 Let $(E_0,E_1)$ be an interpolation couple, $F$ a Banach space and let $0<\theta<1$. Then a linear mapping $L:E_0\cap E_1\rightarrow F$ extends to a continuous operator $L:(E_0,E_1)_{\theta, 1}\rightarrow F$ if and only if there exists a $C>0$ such that for all $u\in E_0\cap E_1$ we have $\norm{Lu}_{E_0 \cap E_1}\leq C\norm{u}_{E_0}^{1-\theta}\norm{u}_{E_1}^{\theta}$.
\end{Lemma}
\begin{proof}[Proof of Proposition \ref{thm: Trace}]
Apply the if-part of Lemma \ref{thm: interpolation lemma} to $E_0=L^p(\Omega)$, $E_1=W^{1,p}(\Omega)$, $F=L^p(\partial\Omega)$, $L=\Gamma$ and $\theta=s$. Use Lemma \ref{thm: trace inequality} to verify the converse. 
\end{proof}

\section{Semiuniform decay of bounded semigroups}\label{apx: semiuniform stability}
We briefly recall three important results connecting resolvent estimates of generators to the decay rate of their corresponding semigroups. In addition to the literature mentioned below we recommend the reader to consult \cite{BattyChillTomilov2016} for a general overview and finer results.

Let $X$ be a Banach space and $B(X)$ the algebra of bounded operators acting on $X$. Throughout this section we assume that $-\A$ is a the generator of a bounded $C_0$-semigroup $T:[0,\infty)\rightarrow B(X)$. By $D(\A), R(A)$ we denote domain and range of $\A$ and by $\sigma(\A)$ its spectrum.

\subsection{Singularity at infinity}
The phrase \emph{``Singularity at infinity''} refers to the situation when the resolvent of $(is+\A)^{-1}$ of $\A$ has no poles on the imaginary axis but is allowed to blow up in operator norm if $s$ tends to infinity. The following theorem is due to Batty and Duyckaerts \cite{BattyDuyckaerts2008} but we also refer to \cite{ChillSeifert2016} for a different proof and to \cite{BattyBorichevTomilov2016} for a generalization. 

\begin{Theorem}[\cite{BattyDuyckaerts2008}]\label{thm: Batty-Duyckaerts}
 Assume that $\sigma(\A)\cap i\R=\emptyset$ and that there exist constants $C, s_0 > 0$ and an increasing function $M:[s_0,\infty)\rightarrow[1,\infty)$ such that 
 \begin{align}\label{eq: resolvent estimate at infinity}
  \forall \abs{s}\geq s_0: \norm{(is+\A)^{-1}}_{X\rightarrow X} \leq C M(C\abs{s}).
 \end{align}
 Then there exist constants $C,t_0 > 0$ such that
 \begin{align*}
  \forall t\geq t_0, \xnice_0 \in D(\A): \norm{T(t)\xnice_0}_{X} \leq \frac{C}{M_{log}^{-1}(\frac{t}{C})}\norm{\xnice_0}_{D(\A)}.
 \end{align*}
 Here $M_{log}(s) = M(s)(\log(1+M(s))+\log(1+s))$.
\end{Theorem}
It is comparatively easy to see that a \emph{semiuniform} decay rate for $T$ as in the conclusion of the Batty-Duyckaerts theorem implies that $\A$ has no spectrum on the imaginary axis. Furthermore if the semigroup decays at least like the decreasing function $m:[0,\infty)\rightarrow(0,\infty)$ then the resolvent can not grow faster than the function $M_1$ at infinity, where $M_1(s)=1+m_r^{-1}(1/(2\abs{s}+1))$ (see \cite[Proposition 1.3]{BattyDuyckaerts2008}). Here $m_r^{-1}$ denotes the right inverse of a decreasing function.

Note that for $M(s)=\abs{s}^{\gamma}$ with $\gamma>0$, this theorem tells us that the decay rate is estimated from above by $(\log(t)/t)^{1/\gamma}$. One may wonder if the logarithmic term is necessary. In general it is, as was shown in \cite{BorichevTomilov2010}, but in the same article one can find the following nice characterization of polynomial decay rates in the Hilbert space setting:

\begin{Theorem}[\cite{BorichevTomilov2010}]\label{thm: Borichev-Tomilov}
 Let $X$ be a Hilbert space and $\gamma>0$. Assume that $\sigma(\A)\cap i\R=\emptyset$ and that there exist constants $C, s_0 > 0$ such that 
 \begin{align*}
  \forall \abs{s}\geq s_0: \norm{(is+\A)^{-1}}_{X\rightarrow X} \leq C \abs{s}^{\gamma}.
 \end{align*}
 Then there exist constants $C,t_0 > 0$ such that
 \begin{align*}
  \forall t\geq t_0, \xnice_0\in D(\A): \norm{T(t)\xnice_0}_{X} \leq C t^{-\frac{1}{\gamma}}\norm{\xnice_0}_{D(\A)}.
 \end{align*}
\end{Theorem}

\subsection{Singularity at zero and infinity}
If $\A$ is our wave operator, by Theorem \ref{thm: Preliminary estimates} (iii), it may happen that $0$ is a spectral point. Therefore it is convenient to have the following generalization of Theorem \ref{thm: Batty-Duyckaerts} at hand:
\begin{Theorem}[\cite{Martinez2011}]\label{thm: Martinez}
 Assume that $\sigma(\A)\cap i\R=\{0\}$. Assume that in addition to the condition (\ref{eq: resolvent estimate at infinity}) there exist $C > 0$, $0<s_1<1$ and a decreasing function $m:(0,s_1)\rightarrow[1,\infty)$ such that 
 \begin{align*}
  \forall \abs{s}\leq s_1: \norm{(is+\A)^{-1}}_{X\rightarrow X} \leq C m(C\abs{s}).
 \end{align*}
 Then there exist constants $C,t_0 > 0$ such that for all $t\geq t_0$ and $\xnice_0\in D(\A)\cap R(\A)$ we have 
 \begin{align*}
  \norm{T(t)\xnice_0}_{X} \leq C\left[\frac{1}{M_{log}^{-1}(\frac{t}{C})} + m_{log}^{-1}(\frac{t}{C}) + \frac{1}{t}\right]\norm{\xnice_0}_{D(\A)\cap R(\A)}.
 \end{align*}
 Here $M_{log}$ is defined as in Theorem \ref{thm: Batty-Duyckaerts} and $m_{log}(s) = m(s)(\log(1+m(s))-\log(s))$.
\end{Theorem}
Concerning the relevance of the $1/t$-term we refer the reader to \cite[Section 8]{BattyChillTomilov2016}. In \cite[Theorem 8.4]{BattyChillTomilov2016} it was shown that in case of polynomial resolvent bounds on a Hilbert space one can get rid of the logarithmic loss.
\end{appendix}


\subsection*{Acknowledgments} I am most grateful to Ralph Chill and Eva Fa\u{s}angov\'a for helpful discussions during my work on the topic of the article and for reading and correcting the first version of this paper. I am also grateful to Otared Kavian for uncovering a gap in the proof of Theorem \ref{thm: Preliminary estimates} related to my ignorance about the borderline case of the trace theorem. Finally I want to thank Lars Perlich for helpful comments.


\bibliographystyle{plainnat}
\bibliography{Refs}

\hfill

Technische Universit\"{a}t Dresden, Fachrichtung Mathematik, Institut f\"{u}r Analysis, 01062, Dresden, Germany. Email: \textit{Reinhard.Stahn@tu-dresden.de}

\end{document}